\RequirePackage{fix-cm}
\documentclass[smallextended]{svjour3}       
\smartqed  
\usepackage{graphicx}
\usepackage[text={6.6in,9.6in},centering]{geometry}
\usepackage{amsmath}
\usepackage{amssymb}
\usepackage{mathrsfs}
\usepackage{comment}
\usepackage{color}
\usepackage{wrapfig}
\usepackage{floatrow}
\usepackage{graphicx,epic,tikz,pifont}
\usepackage{caption}
\usepackage{natbib}
\usepackage{appendix}

\usepackage{hyperref}
\hypersetup{
    colorlinks,%
    citecolor=blue,%
    filecolor=black,%
    linkcolor=blue,%
    urlcolor=blue
}

\usepackage{hhline}

\newcommand{\pp}{\mathbb{P}}
\newcommand{\ee}{\mathbb{E}}

\newcommand{\tsp}{\mathcal{T}}
\newcommand{\rtsp}{\mathcal{T}^{\,*}}
\newcommand{\pe}{E^{\circ}}
\newcommand{\inte}{E^{\bullet}}

\newcommand{\pf}{A}
\newcommand{\ch}{B}

\newcommand{\var}{\mathbb{V}}
\newcommand{\cov}{Cov}
\newcommand{\corr}{\rho}

\newcommand{\YHK}{y}
\newcommand{\PDA}{u}

\newcommand{\epf}{\hfill $\square$}

\newcommand{\pyule}{\mathbb{P}_{\YHK}}
\newcommand{\puni}{\mathbb{P}_{\PDA}}

\newcommand{\pmf}{\sigma}

\newcommand{\ypmf}{\tau} 

\newcommand{\rev}[1]{#1}
\newcommand{\twu}[1]{#1}

\newcommand{\rf}[1]{( #1 )}

\begin{document}

\title{On cherry and pitchfork distributions of random rooted and unrooted phylogenetic trees
\thanks{K.P. Choi acknowledges the support of Singapore Ministry of Education Academic Research Fund [Grant number R-155-000-188-114], and A. Thompson is supported by the UKRI Biotechnology and Biological Sciences Research Council Norwich Research Park Biosciences Doctoral Training Partnership [Grant number BB/M011216/1]. }
}

\titlerunning{On cherry and pitchfork distributions of random trees}        

\author{Kwok Pui Choi        \and
        Ariadne Thompson \and Taoyang Wu
}


\institute{Kwok Pui Choi \at
             Department of Statistics and Applied Probability, and the Department of Mathematics, National University of Singapore, Singapore 117546 \\
              \email{stackp@nus.edu.sg}           
           \and
          Ariadne Thompson and Taoyang Wu \at
              School of Computing Sciences, University of East Anglia, Norwich, NR4 7TJ, U.K. \\
 \email{ariadne.thompson@uea.ac.uk, taoyang.wu@uea.ac.uk}
}

\date{Received: date / Accepted: date}

\maketitle

\begin{abstract}
Tree shape statistics are important for investigating evolutionary mechanisms mediating phylogenetic trees. As a step towards bridging shape statistics between rooted and unrooted trees, we present a comparison study on two subtree statistics known as numbers of cherries and pitchforks for the proportional to distinguishable arrangements (PDA) and  the Yule-Harding-Kingman (YHK) models. Based on recursive formulas on the joint distribution of the number of cherries and that of pitchforks, it is shown that cherry distributions are log-concave for both rooted and unrooted trees under these two models. \twu{Furthermore}, the mean number of cherries and that of pitchforks for unrooted trees converge respectively to those for rooted trees under the YHK model while there exists a limiting gap of $1/4$ for the PDA model. Finally, the total variation distances between the cherry distributions of rooted and those of unrooted trees converge for both models. Our results indicate that caution is required for conducting statistical analysis for tree shapes involving both rooted and unrooted trees.


\end{abstract}

\section{Introduction}

As a common way of representing evolutionary relationships among  biological systems ranging from genes to populations, phylogenetic trees retain important signatures of the underlying evolutionary events and mechanisms which are often not directly observable, such as  speciation and expansion~\citep{mooers2007some,heath2008taxon}.
To utilise these signatures, one popular approach is to compare empirical tree shape indices with those predicted by neutral models specifying tree generating processes
~\citep[see, e.g.][]{blum2006random,colijn2014phylogenetic,hagen2015age}. Moreover, topological tree shapes are also closely related to a number of basic population genetic statistics~\citep{Ferretti2017,arbisser2018joint} and
\twu{are} useful for identifying the \rev{bias} of tree reconstruction methods~\citep{PickettRandle05,holton2014shape}.

Phylogenetic trees can be broadly grouped into two categories: rooted and unrooted.
A rooted tree \twu{contains} a specific node designated as the root, which gives a temporal direction to the tree that unrooted trees do not have.
Tree shapes for rooted trees are relatively better studied, partly due to  a general framework known as \twu{the} recursive shape index ~\citep{matsen2007optimization,chang2010limit,disanto2013exact,cardona2013exact}.
However, less is \rev{known} about properties of tree shapes for unrooted trees, which are  used extensively in phylogenetic analysis, particularly when the location of the root is too difficult to be inferred from data~\citep{steel2012root}.
As a step towards bridging this gap, we investigate the exact joint distribution of two subtree statistics known as cherries and pitchforks, e.g. subtrees of two and three leaves, \rev{respectively}, for unrooted trees under two commonly used tree generating models: the proportional to distinguishable arrangements (PDA) and    Yule-Harding-Kingman (YHK) models.
 Combining \twu{these results} with subtree distributions on rooted trees~\citep{McKenzie2000,rosenberg06a,WuChoi16}, we then conduct a comparison study on these subtree distributions between rooted and unrooted trees to gain further insights into these two models.

We now summarize the contents of the rest of the paper. In the next section, we begin by reviewing some definitions and results concerning phylogenetic trees and tree-generating models. These models are described using a random tree growth framework based on taxon attachment that can be applied to both rooted and unrooted trees. \twu{We then} study the cherry and the pitchfork distributions under the PDA model in Section~\ref{sec:results_pda} and those under the YHK model in Section~\ref{sec:results_yhk}. For each model, our starting point is a recursive formula on the joint distributions of cherries and pitchforks, which leads to a common approach to  computing the mean and variance for cherries and for pitchforks, as well as their covariance.

In Section~\ref{sec:comparison} we present comparison studies on properties of the cherry and pitchfork distributions under the two models for both rooted and unrooted cases. In particular, \rev{it is shown} that the cherry distributions under both models are log-concave (Theorem~\ref{them:log-concav}).  \twu{For both cherry distributions and pitchfork distributions, the difference between the mean number for unrooted trees and that for rooted trees converges to $0$ under the YHK model, whereas it converges to $1/4$ under the PDA model (Proposition~\ref{prop:mean:comp})}.
Finally, we show that the total variation distances between the cherry distributions of rooted and those of unrooted trees \twu{converge under the YHK model~(Proposition~\ref{prop:yule:tv}), and converge to $0$ under the PDA model
~(Theorem~\ref{thm:pda_cherry_dtv})}. We conclude in the last section with a discussion of our results and some open problems.


\section{Preliminaries} \label{sec:preliminaries}

In this section, we present some basic notation and background concerning phylogenetic trees, random tree models, and log-concavity. From now on $n$ will be a positive integer greater than three unless stated otherwise and $[n]=\{1,2,\dots,n\}$ is the set of all positive integers between 1 and $n$.

\subsection{Phylogenetic trees}
 A {\em tree} $T=\left( V(T),E(T) \right)$ is a connected acyclic graph with vertex set $V(T)$ and edge set $E(T)$. A vertex is referred to as a {\em leaf} if it has degree one, and an {\em interior vertex} otherwise.  An edge incident to a leaf is called a {\em pendant edge}, and let $\pe(T)$ be the set of pendant edges in $T$.
  A tree is {\em rooted} if it contains exactly one distinguished \twu{degree-one} node designated as the {\em root}, which is not regarded as a leaf and is usually denoted by $\rho$, and {\em unrooted} otherwise.  All trees consider here are {\em binary}, that is, each interior vertex has degree three.

 A (rooted) {\em phylogenetic tree} on a finite set $X$ is a (rooted)  tree with leaves bijectively labelled by the elements of $X$. The set of binary unrooted phylogenetic trees and rooted phylogenetic trees on $[n]$ are denoted by $\tsp_n$ and $\rtsp_n$, respectively.
 See Fig.~\ref{fig:tree_examples} for examples of trees in $\tsp_8$ and $\rtsp_8$.

Removing the root of a tree $T^*$ in $\rtsp_n$, that is, removing $\rho$ and suppressing its adjacent interior vertex $r$ (i.e. deleting $r$ and adding an edge connecting the remaining two neighbours of $r$), results in an unrooted tree $\rho^{-1}(T^*)$ in $\tsp_n$.
 For example,  \rev{in Fig.~\ref{fig:tree_examples} the unrooted} phylogenetic tree $T$  is obtained from the rooted \rev{tree $T^*$} by removing its root.

 \begin{figure}[ht]
\begin{center}
{\includegraphics[scale=0.8]{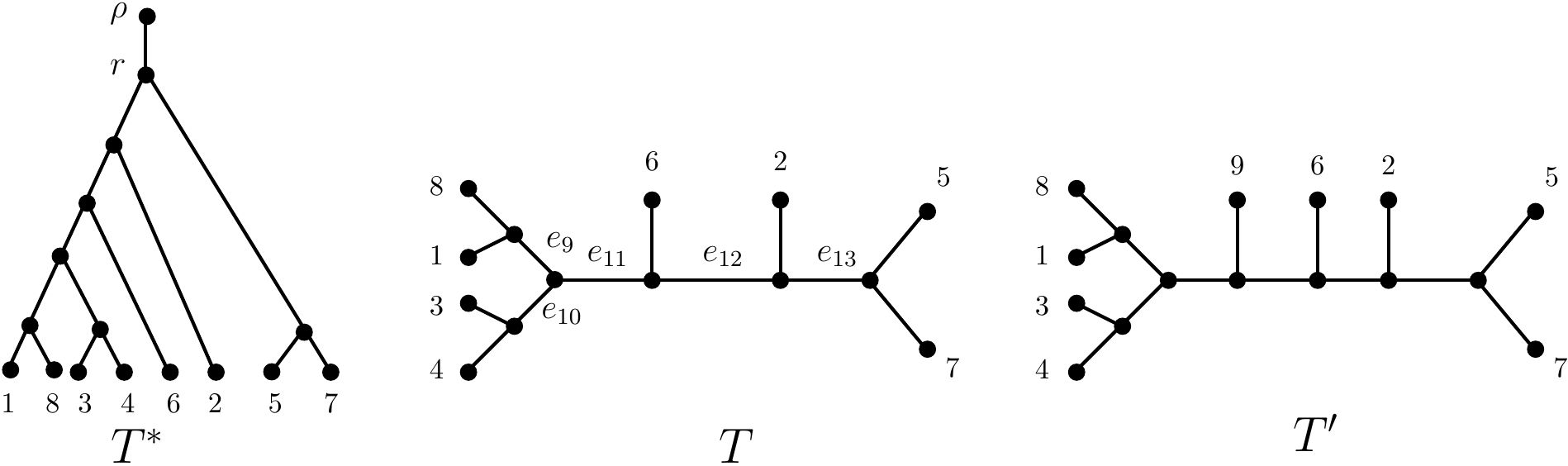}}
\end{center}
\caption{Examples of phylogenetic trees.  $T^*$ is a rooted phylogenetic tree on $X=\{1,\dots,8\}$; $T$ is an unrooted phylogenetic tree on $X$ (where the pendant edge incident to $i$ is $e_i$); $T'=T[e_{11}]$ is a phylogenetic tree on $\{1,\dots,9\}$ obtained from $T$ by attaching a new leaf labelled $9$ to edge $e_{11}$.
}
\label{fig:tree_examples}
\end{figure}

Removing an edge in a phylogenetic tree $T$ results in two connected components; such a connected component is referred to as a subtree of $T$ if it does not contain the root of $T$. In other words, removing an edge in a phylogenetic tree $T$ results in two subtrees if $T$ is unrooted, and one subtree if $T$ is rooted. A subtree is called a {\em cherry} if it has two leaves, and a {\em pitchfork} if it has three leaves.  Given a phylogenetic tree $T$, let $\pf(T)$ and $\ch(T)$ be the number of  pitchforks and cherries, \twu{respectively}, contained in $T$.

Given an edge $e$ in a phylogenetic tree $T$ on $X$ and a taxon $x' \not \in X$, let $T[e;x']$ be the phylogenetic tree on $X\cup \{x'\}$ obtained by attaching a new leaf with label $x'$ to the edge $e$. Formally, let $e=\{u,v\}$ and let $w$ be a vertex not contained in $V(T)$. Then $T[e;x']$ has vertex set $V(T) \cup \{x',w\}$ and edge set $\big(E(T) \setminus \{e\} \big) \cup \{(u,w), (v,w), (w,x')\}$. See Fig.~\ref{fig:tree_examples} for an illustration of this construction, \rev{where tree $T'=T[e_{11};9]$ is obtained from $T$ by attaching leaf $9$ to edge $e_{11}$}. Note that we also use $T[e]$ instead of $T[e;x']$ when the taxon name $x'$ is not essential.

\subsection{The YHK and the PDA models}
\label{subsection:model}
Let $\tsp_n$ be the set of unrooted phylogenetic trees with $n$ leaves.
In this subsection, we present a formal definition of the two null models investigated in this paper:  the  proportional to distinguishable arrangements (PDA) model and the
Yule-Harding-Kingman (YHK) model.  Although these two models are typically presented through the rooted version, for our purpose here we follow the direct approach based on Markov processes as described by~\citet[Section 3.3.3]{steel2016phylogeny}.

Under the YHK process~\citep{harding71a, yule25a}, a  random unrooted phylogenetic tree $T_n$ is generated as follows.
\begin{itemize}
\item[(i)] Select a uniform random permutation $(x_1,\dots,x_n)$ of $[n]$;
\item[(ii)] start with the unrooted phylogenetic tree $T_2$ on the taxon set $\{x_1,x_2\}$;
\item[(iii)] for $2\le k <n$, uniformly choose a random pendant edge $e$ in $T_k$ and let
$T_{k+1}=T_k[e;x_{k+1}]$.
\end{itemize}
Here a permutation $(x_1,\dots,x_n)$ of $[n]$ means a taxon sequence with $x_i \in [n]$ and $x_i\not =x_j$ for all $i\not =j$.
A similar process can be used to sample a rooted tree under the YHK model by using the rooted phylogenetic tree $T_2$ on $\{x_1,x_2\}$ in Step (ii).
Moreover, the PDA process can be described using a similar scheme; the only difference is that in Step (iii) the edge $e$ is uniformly sampled from the edge set of $T_k$, instead of the pendant edge set\footnote{Under the PDA process, Step (i) can also be simplified by using a fixed permutation, say $(1,2,\cdots,n)$.}.

The probability of generating a given unrooted tree $T$ under the YHK model (respectively the PDA model) is denoted by $\pyule(T)$ (resp. $\pp _{\PDA}(T)$).
We use $\ee _{\YHK}$, $\var _{\YHK}$, $\cov _{\YHK}$, and $\corr _{\YHK}$ to  denote respectively  the  expectation, variance, covariance, and correlation taken with respect to the probability measure $\pp _{\YHK}$ under the YHK model. Similarly, $\ee _{\PDA}$, $\var _{\PDA}$, $\cov _{\PDA}$, and $\corr _{\PDA}$ are defined with respect to the probability
measure $\pp _{\PDA}$ under the PDA model.

\rev{For $n\ge 4$, let $\ch_n$ be the random variable $\ch(T)$ for a random tree $T$ in $\tsp_n$. Similarly, for $n\ge 6$,   let $\pf_n$ be the random variable $\pf(T)$ for a random tree $T$ in $\tsp_n$.  The probability distributions of $\pf_n$ (resp. $\ch_n$) will be referred to as pitchfork distributions (resp. cherry distributions).
Since each tree in $\tsp_n$ with $n=4,5$ contains precisely two cherries, the following statement clearly holds.
\begin{equation}
\label{eq:ch:four:five}
~~\mbox{For $n=4,5$, we have}~\quad
\pyule(\ch_n=b)
=\pp_{\PDA}(\ch_n=b)=\begin{cases}
1, & \text{if $b=2$}, \\
0, & \text{otherwise}.
\end{cases}
\end{equation}
}
\rev{When $n=6$, we have
\begin{equation}
\label{eq:six:distribution}
~\quad
\pyule(\pf_6=a,\ch_6=b)
=\begin{cases}
\frac{4}{5}, & \text{if $a=2,b=2$}, \\
\frac{1}{5}, & \text{if $a=0,b=3$}, \\
0, & \text{otherwise},
\end{cases}
\quad\mbox{and}\quad
\pp_{\PDA}(\pf_6=a,\ch_6=b)
=\begin{cases}
\frac{6}{7}, & \text{if $a=2,b=2$}, \\
\frac{1}{7}, & \text{if $a=0,b=3$}, \\
0, & \text{otherwise}.
\end{cases}
\end{equation}
}

\noindent In  this paper, we are interested in the joint  and marginal distributional properties of $\pf_n$ and $\ch_n$ under the YHK and the PDA models \rev{for $n\ge 6$}.


\subsection{Edge decomposition}
The edge set $E(T)$ of a phylogenetic tree $T$ can be decomposed into two sets: the pendant edge set $\pe(T)$ and the interior edge set $\inte(T)$. That is, we have
$$
E(T)= \pe(T) \sqcup \inte(T),
$$
where $\sqcup$ denotes the union of disjoint sets.  A cherry is called {\em essential} if it is not contained in a pitchfork.  Let $\inte_{ec}(T)$ be the set of interior edges whose removal induces an essential cherry \rev{as one of the two subtrees}, and denote its complement in $\inte(T)$ by $\inte_{nec}(T)$. Then we have
$$
\inte(T)= \inte_{ec}(T) \sqcup \inte_{nec}(T).
$$
For \rev{pendant edges}, let $\pe_{ec}(T)$ be the set of pendant edges contained in essential cherries, let $\pe_{pf}(T)$ be the set of pendant edges contained in a pitchfork but not a cherry, let $\pe_{cp}$ be the set of pendant edges contained both in a cherry and in a pitchfork, let $\pe_{ind}$ be the set of pendant edges contained in neither a cherry nor a pitchfork.
To illustrate the above notation, considering the tree $T$ depicted  in Fig.~\ref{fig:tree_examples}, we have $\pe_{ec}(T)=\{e_1,e_3,e_4,e_8\}$, $\pe_{pf}(T)=\{e_2\}$, $\pe_{cp}(T)=\{e_5,e_7\}$,
$\pe_{ind}(T)=\{e_6\}$, $\inte_{ec}=\{e_9,e_{10}\}$, and $\inte_{nec}(T)=\{e_{11},e_{12},e_{13}\}$.

\rev{
For a tree $T$ in $\tsp_n$ with $n\ge 6$, we have
\begin{align}
	  \pe(T) &= {} \pe_{ec}(T) \sqcup \pe_{pf}(T) \sqcup \pe_{cp}(T) \sqcup \pe_{ind}(T)~~\mbox{and}~~ \label{eq:edge:dec} \\
   E(T) &={} \inte_{ec}(T) \sqcup \inte_{nec}(T) \sqcup \pe_{ec}(T) \sqcup \pe_{pf}(T) \sqcup \pe_{cp}(T) \sqcup \pe_{ind}(T).
\end{align}
Furthermore,  it is easy to see that $|E(T)|=2n-3$, $|\pe_{pf}(T)|=A(T)$, and $|\pe_{cp}(T)|=2A(T)$.} Since each pitchfork contains precisely one non-essential cherry, it follows that $|\pe_{ec}(T)|=2(B(T)-A(T))$ and $|\inte_{ec}(T)|=B(T)-A(T)$.  Finally, we have $|\pe_{ind}(T)|=n-A(T)-2B(T)$ and
$|\inte_{nec}(T)|=n-3+A(T)-B(T)$ because $T$ has precisely $n$ pendant edges and the number of interior edges in it is $n-3$.

The following lemma, whose proof is straightforward and hence omitted here, relates the values $\ch(T[e])-\ch(T)$ and $\pf(T[e])-\pf(T)$ to the choice of $e$.

\begin{lemma} \label{prop:yhk_growth_cases}
Suppose that $e$ is an edge in a phylogenetic tree $T$ and $T'=T[e]$. Then we have
\small{
\begin{equation*}
\pf(T') = \begin{cases}
             \pf(T)-1  & \text{if } e\in \pe_{pf}(T),\\
             \pf(T)+1 & \text{if } e\in \pe_{ec}(T) \cup \inte_{ec}(T), \\
              \pf(T)  & \text{otherwise};
       	  \end{cases}
      	 \quad \mbox{and} \quad
\ch(T') = \begin{cases}
             \ch(T)+1 & \text{if } e \in \pe_{pf}(T)\cup \pe_{ind}(T), \\
             {} & \\
             \ch(T) & \text{otherwise }.
       	  \end{cases}
\end{equation*}
}
\end{lemma}

\subsection{Miscellaneous}

A sequence $(y_1, \dots, y_m)$ of numbers is called {\em positive} if each number in the sequence is positive, that is, greater than zero. It is defined to be {\em log-concave} if $y_{k-1} y_{k+1} \leq y_k^2$ holds for $2 \leq k \leq m-1$.  Clearly, a positive sequence $(y_k)_{1 \leq k \leq m}$ is log-concave if and only if the sequence $(y_{k} / y_{k+1})_{1 \leq k \leq m-1}$ is increasing. Therefore, a log-concave sequence is necessarily {\em unimodal}; that is, there exists an index $1\leq k \leq m$ such that
\begin{equation*}
	y_1\leq y_2 \leq \dots \leq y_k \geq y_{k+1} \geq \cdots \geq y_m
\end{equation*}

\noindent
holds. Moreover, a non-negative integer-valued random variable $Y$ with probability mass function $\{p_k: k\geq 0\}$  is log-concave if $\{p_k\}_{k\geq 0}$ is a log-concave sequence.

Next, for each positive integer $n$, we let $\Delta(n)=\lfloor n/4 \rfloor$ be the largest integer that is less than or equal to $n/4$. Then we have
\begin{equation}
\label{eq:delta}
\left \lfloor \frac{ n(n - 1) }{ 2(2n - 3) } \right \rfloor =\Delta(n)~~\mbox{for $n\ge 4$.}
\end{equation}
To see \twu{that} the last equation holds, let $r$ be the integer between $0$ and $3$ such that $n=4m+r$ holds for some $m\ge 1$. When $r=0$, that is, $n=4m$ for some $m\ge 1$, it is straightforward to verify that  $m<4m(4m-1)/(16m-6)<m+1$, from which~(\ref{eq:delta}) follows. The other three cases, where $r\in \{1,2,3\}$, can be verified in a similar manner. \rev{Furthermore}, let $\nabla(n)=\lceil n/4 \rceil$ be the smallest integer that is greater than or equal to $n/4$. Then a proof similar to that of~(\ref{eq:delta}) shows that
\begin{equation}
\label{eq:nabla}
\left \lfloor \frac{(n + 1) (n + 2)}{2 (2n - 1)}  \right \rfloor =\nabla(n)~~\mbox{for $n>8$}.
\end{equation}

We end this section with \twu{the following} fact on the rising factorial.   Let $k^{\rf{r}}=k(k+1)\cdots (k+r-1)$ for positive \rev{integers} $k,r$. \twu{Then we have}  the following identity~\citep[see, e.g.][]{RKP94}:
\begin{equation} \label{eq:telescopic_sum}
	\sum_{k=1}^m k^{\rf{r}} = \frac{m^{\rf{r+1}}}{r+1}.
\end{equation}
When $r = 1$, this gives us the following well-known formula for the triangular numbers:
\begin{equation} \label{eq:triangular_numbers}
	1 + 2 + \cdots + m = \frac{m^{\rf{2}}}{2} = \frac{m(m+1)}{2}.
\end{equation}


\section{Subtree distributions under the PDA model}
 \label{sec:results_pda}

\rev{In this section, we study the joint distribution of the random variables $A_n$ (i.e. the number of pitchforks) and $B_n$ (i.e., the number of cherries) under the PDA model. }

\rev{Our starting point is the following result on a} recursion of the joint distribution, which will then be used to deduce the marginal distribution of $B_n$, as well as the joint moments of $A_n$ and $B_n$.

\begin{theorem}
\label{thm:PDA:pro:rec}
\rev{Let $n\ge 6$. Then we have $\puni(\pf_{n}=a, \ch_{n}=b)=0$ if either $a\in \{-1,n,n+1\}$ or $b\in \{1,n\}$ holds. Furthermore, for $0\le a \le n$ and $1< b \le n$ we have}
\begin{align}
\puni&(\pf_{n+1}=a,\ch_{n+1}=b)
=\frac{n+3a-b-3}{2n-3}\puni(\pf_{n}=a,\ch_n=b)+\frac{a+1}{2n-3}\puni(\pf_{n}=a+1, \ch_n =b-1)  \nonumber \\
&\quad +\frac{3(b-a+1)}{2n-3}\puni(\pf_{n}=a-1, \ch_n=b)+\frac{n-a-2b+2}{2n-3}\puni(\pf_{n}=a, \ch_n=b-1).  \label{eq:rec:pda}
\end{align}
\end{theorem}

\begin{proof}
\rev{Fix an integer $n\ge 6$. Since $0\le \pf(T)\le n/3$ and $1<\ch(T)\le n/2$ holds for every tree $T\in \tsp_n$, it follows that $\puni(\pf_{n}=a, \ch_{n}=b)=0$ if either $a\in \{-1,n,n+1\}$ or $b\in \{1,n\}$ holds.}

Next, let $T_2, \dots, T_n, T_{n+1}$ be a sequence of random trees generated by the PDA process. That is, \rev{choosing a random permutation $(x_1,\dots,x_{n+1})$ of $[n+1]$ and considering the tree $T_2$ with two leaves $\{x_1,x_2\}$,  then} $T_{i+1} = T_i[e_i;x_{i+1}]$ where $e_i$ is  a uniformly chosen edge in $T_i$ for $2\le i \le n$. In particular, we have $|E(T_i)| = 2i-3$  for $2 \le i \le  n+1$. \rev{For $0\le a \le n$ and $1< b \le n$} put $\Delta(a,b)=\{(a,b),(a+1,b-1),(a-1,b),(a,b-1)\}$. Then we have
\begin{align}
&\puni(A_{n+1} = a, B_{n+1} = b) = \pp (A(T_{n+1}) = a, B(T_{n+1}) = b) \nonumber \\
&\,\, = \sum_{p,q} \pp(A(T_{n+1}) = a,B(T_{n+1}) = b | A(T_n) = p, B(T_n) = q)\pp(A(T_n) = p, B(T_n) = q) \nonumber \\
&\,\, = \sum_{p,q} \pp(A(T_{n+1}) = a,B(T_{n+1}) = b | A(T_n) = p, B(T_n) = q)\puni(A_n = p,B_n = q) \nonumber \\
&\,\,= \sum_{(p,q)\in \Delta(a,b)} \pp(A(T_{n+1}) = a,B(T_{n+1}) = b | A(T_n) = p, B(T_n) = q)\puni(A_n = p,B_n = q). \label{eq:thm:puni}
\end{align}
Here the first and third equalities follow \rev{from} the definition of random variables $A_n$ and
$B_n$ while the second one \rev{follows} from the law of total probability. The last equality holds because by Lemma~\ref{prop:yhk_growth_cases} and $T_{n+1}=T_{n}[e_n;x_{n+1}]$ we have
$$
 \pp(A(T_{n+1}) = a,B(T_{n+1}) = b | A(T_n) = p, B(T_n) = q) = 0
 $$
 for $(p,q)\not \in \Delta(a,b)$.

It remains to consider the cases with $(p,q)\in \Delta(a,b)$. The first case is that $(p,q)=(a,b)$. By Lemma~\ref{prop:yhk_growth_cases}, we have
\begin{align}
\label{eq:case1}
 \pp(A(T_{n+1}) = a,B(T_{n+1}) = b | A(T_n) = a, B(T_n) = b) = \frac{|\inte_{nec}(T_n)\cup \pe_{cp}(T_n)|}{E(T_n)} =\frac{n+3a-b-3}{2n-3}.
\end{align}
 Similarly, we have
  \begin{align}
&   \pp(A(T_{n+1}) = a,B(T_{n+1}) = b | A(T_n) = a+1, B(T_n) = b-1) = \frac{|\pe_{pf}(T_n)|}{E(T_n)} =\frac{a+1}{2n-3}, \label{eq:case2}\\
&      \pp(A(T_{n+1}) = a,B(T_{n+1}) = b | A(T_n) = a-1, B(T_n) = b) = \frac{|\pe_{ec}(T_n)\cup \inte_{ec}(T_n)|}{E(T_n)} =\frac{3(b-a+1)}{2n-3}, \label{eq:case3} \\
&         \pp(A(T_{n+1}) = a,B(T_{n+1}) = b | A(T_n) = a, B(T_n) = b-1) = \frac{|\pe_{ind}(T_n)|}{E(T_n)} =\frac{n-a-2b+2}{2n-3}. \label{eq:case4}
  \end{align}

Now the theorem follows from substituting~(\ref{eq:case1})-(\ref{eq:case4})
\rev{into}~(\ref{eq:thm:puni}).
\epf
\end{proof}

Similar to the dynamic programming approach outlined in~\citet[p.16]{WuChoi16},  we can use \rev{the initial condition in~\eqref{eq:six:distribution} and the} recursion in \rev{Theorem~\ref{thm:PDA:pro:rec}} to compute the joint distribution of $\pf_n$ and $\ch_n$ under the PDA model in $O(n^3)$. \rev{See Fig.~\ref{fig:probability_contours} for the contour plots of the probability density functions for the joint distribution of the numbers of cherries and pitchforks on unrooted phylogenetic trees with 50 and 200 leaves.
}

\begin{figure}
	\centering
    \begin{minipage}{.49\textwidth}
    	\centering
    	\includegraphics[height=0.6\linewidth]{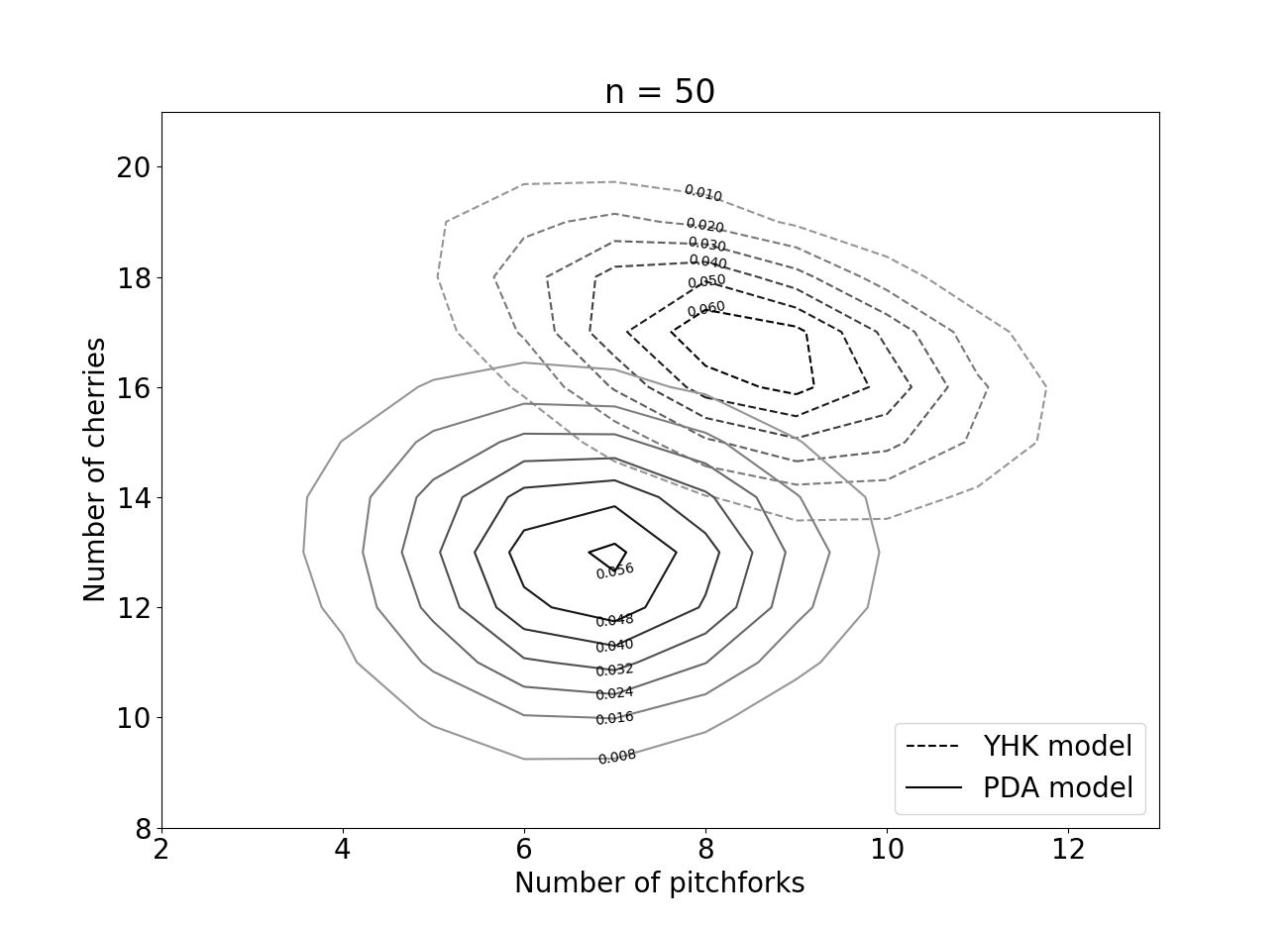}
    \end{minipage}
    \begin{minipage}{.49\textwidth}
    	\centering
    	\includegraphics[height=0.6\linewidth]{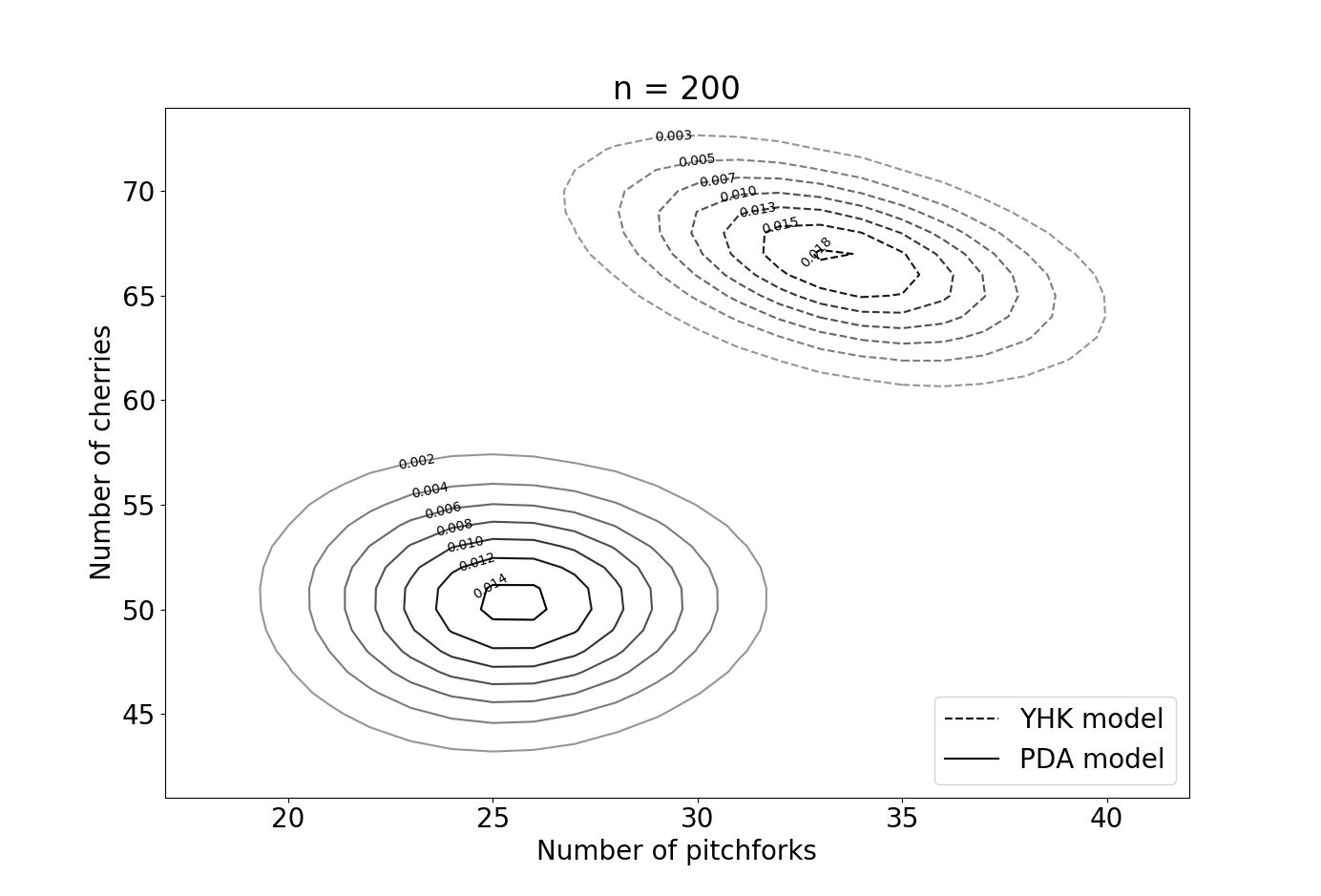}
    \end{minipage}
    \caption{Contour plots of the probability density functions for the joint distribution of the numbers of cherries and pitchforks on unrooted phylogenetic trees under the PDA model (solid lines) and the YHK model (dashed lines) with 50 leaves (left) and 200 leaves (right). The polygonal contours arise because the joint distribution is defined only on integer lattice points.
    }
    \label{fig:probability_contours}
\end{figure}

To study the moments of $\pf_n$ and $\ch_n$, we  present below a functional recursion form of Theorem~\ref{thm:PDA:pro:rec}.

\begin{theorem} \label{thm:pda_arbitrary_function}
Let $\varphi: \mathbb{N}  \times \mathbb{N}  \to \mathbb{R} $ be an arbitrary function. For $n \ge 6$ we have
\begin{align}
 \ee _{\PDA} \varphi(\pf_{n+1}, \ch_{n+1}) & = \frac{1}{2n-3} \ee _{\PDA}[(n+3\pf_{n}-\ch_n-3) \: \varphi (\pf_{n}, \ch_{n})] + \frac{1}{2n-3} \ee _{\PDA}[\pf_{n} \: \varphi (\pf_{n}-1, \ch_{n}+1)] \nonumber  \\
& + \frac{3}{2n-3} \ee _{\PDA}[ (\ch_{n}-\pf_n) \: \varphi (\pf_{n}+1, \ch_{n})]  + \frac{1}{2n-3}  \ee _{\PDA}[(n-\pf_{n}-2\ch_{n}) \: \varphi (\pf_{n}, \ch_{n}+1)].
\end{align}
\end{theorem}

\begin{proof}
Fix two arbitrary numbers $a,b \in \mathbb{N}$ and consider the indicator function $I_{a,b}: \mathbb{R} \times \mathbb{R} \to \{0,1\}$  where $I_{a,b}(x,y)=1$ if and only if $x=a,y=b$. \twu{Multiplying both sides of~(\ref{eq:rec:pda}) by $\varphi(a,b)$ leads} to
{\small
\begin{align*}
(2n-3)&\ee _{\PDA}[ \varphi(\pf_{n+1},\ch_{n+1}) I_{a,b}(\pf_{n+1}, \ch_{n+1})]
=\ee _{\PDA}[ (n+3\pf_n-\ch_n-3)\varphi(\pf_n,\ch_n) I_{a,b}(\pf_{n}, \ch_{n})]  \\
&+\ee _{\PDA}[ \pf_n \varphi(\pf_n-1,\ch_n+1) I_{a,b}(\pf_{n}-1, \ch_{n}+1)]
+\ee _{\PDA}[ 3(\ch_n-\pf_n) \varphi(\pf_n+1,\ch_n) I_{a,b}(\pf_{n}+1, \ch_{n})]
\\
&+\ee _{\PDA}[ (n-\pf_n-2\ch_n) \varphi(\pf_n,\ch_n+1) I_{a,b}(\pf_{n}, \ch_{n}+1)] .
\end{align*}
}
\noindent
Now the theorem follows from summing over all $a$ and $b$.
\epf
\end{proof}

\rev{Theorem~\ref{thm:pda_arbitrary_function} leads to the following result on cherry distributions.}

\begin{proposition} \label{prop:pda_cherry_arbitrary_function}

Let $\psi: \mathbb{N} \to \mathbb{R}$ be an arbitrary function. Then for \rev{$n \ge 4$,} we have
\begin{align} \label{eq:pda_cherry_arbitrary_function}
\ee _{\PDA} &\psi(\ch_{n+1}) = \frac{1}{2n-3} \ee _{\PDA}[(n+2\ch_n-3) \: \psi (\ch_{n})] + \frac{1}{2n-3} \ee _{\PDA}[(n-2\ch_{n}) \: \psi (\ch_{n}+1)].
\end{align}

\end{proposition}

\begin{proof}
Let $\psi: \mathbb{N} \to \mathbb{R}$ be an arbitrary function as in the statement of the proposition. Then~\eqref{eq:pda_cherry_arbitrary_function} clearly holds for $n=4,5$ in view of \twu{\eqref{eq:ch:four:five} and~\eqref{eq:six:distribution}.} 
Now consider the function $\varphi^*(x,y)=\psi(y)$ on $\mathbb{N} \times \mathbb{N}$. Applying Theorem~\ref{thm:pda_arbitrary_function} to the function $\varphi^*$ shows that~(\ref{eq:pda_cherry_arbitrary_function}) holds for $n\ge 6$.
\epf
\end{proof}

Note that the recursion in~\rev{Proposition~\ref{prop:pda_cherry_arbitrary_function} can be utilised } to study the moments of \rev{the} cherry distribution under the PDA model. As an example, we present below an alternative proof for the well-known result on the mean and variance of $\ch_n$ obtained by~\cite{McKenzie2000,steel93distr}.

\begin{corollary}
\label{cor:pda:ch}
For $n\geq 4$,  we have
\begin{eqnarray}
	\ee _{\PDA}(\ch_n) &=& \frac{n(n - 1)}{2(2n - 5)} = \frac{n}{4} + \frac{3}{8} + O(n^{-1}), \label{eq:pda_cherry_mean} \\
	\var _{\PDA}(\ch_n) &=& \frac{n(n - 1)(n - 4)(n - 5)}{2(2n - 5)^2(2n - 7)} = \frac{n}{16} - \frac{3}{32} + O(n^{-1}). \label{eq:pda_cherry_variance}
\end{eqnarray}

\end{corollary}

\begin{proof}
Note that the corollary clearly holds for $n=4$ because by~\eqref{eq:ch:four:five} we have $\ee _{\PDA}(\ch_n)=2$ and $\var _{\PDA}(\ch_n)=0$.

Substituting $\psi(x) = x$ in the recursive equation~(\ref{eq:pda_cherry_arbitrary_function}) in Proposition~\ref{prop:pda_cherry_arbitrary_function} shows that
\begin{equation}
\label{eq:pda:ch:rec}
\ee _{\PDA}(\ch_{n+1}) = \frac{1}{2n-3}\ee _{\PDA}\big[(n + 2\ch_n - 3)\ch_n + (n - 2\ch_n)(\ch_n + 1)\big] = \frac{n}{2n - 3} + \frac{2n - 5}{2n - 3} \ee _{\PDA}(\ch_n)\,
\end{equation}
holds for $n\ge 4$.

The recurrence in~\eqref{eq:pda:ch:rec} can be solved by a summation factor method. First, we multiply both sides of~\eqref{eq:pda:ch:rec} by the summation factor $s_n=2n-3$. Next, set $f(n)= s_{n-1}\ee _{\PDA}(\ch_n)=(2n - 5) \ee _{\PDA}(\ch_n)$ for $n\ge 4$. 
Then substituting $n$ with $n-1$ in~\eqref{eq:pda:ch:rec} leads to
\begin{equation*}
	f(n) = (n - 1) +f(n - 1)~~\mbox{for $n\ge 5$}.
\end{equation*}
Finally, solving the recurrence on $f(n)$ gives us that
\begin{equation}
\label{eq:ch:mean:pda:recur}
f(n) =	\twu{ \sum_{k=5}^{n}(k-1)+f(4)=}\sum_{k=1}^{n}(k-1)=\frac{n (n - 1)}{2}~~\mbox{for $n\ge 5$},
\end{equation}
from which~(\ref{eq:pda_cherry_mean}) follows. \twu{Note that the second equality in~\eqref{eq:ch:mean:pda:recur}}  follows from the fact \twu{that} $f(4)=6$,  and the third equality follows from~\eqref{eq:triangular_numbers}.

To establish~\eqref{eq:pda_cherry_variance}, substituting $\psi(x) = x^2$ into~(\ref{eq:pda_cherry_arbitrary_function}) and using~\eqref{eq:pda_cherry_mean} shows that
\begin{eqnarray}
	\ee _{\PDA}(\ch_{n+1}^2) &=& \frac{2n - 7}{2n - 3} \ee _{\PDA}(\ch_{n}^2) + \frac{2 (n - 1)}{2n - 3} \ee _{\PDA}(\ch_{n}) + \frac{n}{2n - 3} \notag \\
    &=& \frac{2n - 7}{2n - 3} \ee _{\PDA}(\ch_{n}^2) + \frac{n (n - 1)^2}{(2n - 3) (2n - 5)} + \frac{n}{2n - 3} \label{eq:pda:pf:rec}
\end{eqnarray}
holds for $n\ge 4$. Similarly to the proof of~\eqref{eq:pda_cherry_mean},  the recurrence in~\eqref{eq:pda:pf:rec} can be solved using the summation factor method. That is, consider the summation factor $s^*_n=(2n-3)(2n-5)$ and set $f^*(n)= s^*_{n-1}\ee _{\PDA}(\ch^2_n)=(2n - 5)(2n-7) \ee _{\PDA}(\ch^2_n)$.
Then for $n\ge 5$, by~\eqref{eq:pda:pf:rec} we have
\begin{eqnarray*}
f^*(n) &=& f^*(n-1) + (n - 1) (n - 2)^2 + (n - 1) (2n - 7)
    = f^*(4) + \sum_{k=3}^{n-2}k^{(3)}-3\sum_{k=4}^{n-1}k \\
    &=& \sum_{k=1}^{n-2}k^{(3)}-3\sum_{k=1}^{n-1}k
    = \frac{(n-2)^{(4)}}{4}-\frac{3(n-1)n}{2}
    = \frac{(n-1)n(n^2-n-8)}{4}.
\end{eqnarray*}
Here the third equality follows from $f^*(4)=12$ and the fourth one follows from~\eqref{eq:telescopic_sum}.
This implies
\begin{equation}
\label{eq:pda_cherry_square}
	\ee _{\PDA}(\ch_{n}^2) = \frac{1}{4} \frac{n (n-1)(n^2 - n -8)}{(2n - 5) (2n - 7)},
\end{equation}
\noindent
and hence~(\ref{eq:pda_cherry_variance}) follows in view of~\eqref{eq:pda_cherry_mean}.
\epf
\end{proof}

Using the recursion in Proposition~\ref{prop:pda_cherry_arbitrary_function} also leads to an alternative proof of the following exact formula on \rev{the} cherry distribution for the PDA model obtained by~\cite{hendy1982branch}.

\begin{theorem} \label{thm:pda_cherry_probability_distribution}
For $n\ge 4$, we have
\begin{equation} \label{eq:pda_cherry_probability}
\pp _{\PDA}(\ch_{n+1}=k) = \frac{n+2k-3}{2n-3}\pp _{\PDA}(\ch_{n}=k)+\frac{n-2k+2}{2n-3}\pp _{\PDA}(\ch_{n}=k-1)~\mbox{for $1<k<n$.}
\end{equation}
Furthermore, we have
\begin{equation}
\pp _{\PDA}(\ch_n=k)=
\frac{n!(n-2)!(n-4)!2^{n-2k}}{(n-2k)!(2n-4)!k!(k-2)!}~~~\mbox{for~$2\leq k\leq n/2$}.\label{eq:pda_cherry_pd}
\end{equation}
\end{theorem}

\begin{proof}
First, by~\eqref{eq:ch:four:five} the theorem clearly holds for $n=4$.
Next, consider the function $I_k(x)$ which equals $1$ if $x=k$, and $0$ otherwise. By taking $\psi(x)=I_k(x)$ in~(\ref{eq:pda_cherry_arbitrary_function}) shows that
\eqref{eq:pda_cherry_probability} holds for $n\ge 4$ and $1<k<n$. Together with $\pp _{\PDA}(\ch_{n}=1)=0$ for $n\ge 4$, it is straightforward to show that \rev{(\ref{eq:pda_cherry_pd})} holds for $n> 4 $ and $2\le k \le n/2$.
\epf
\end{proof}

We complete this section by using Theorem~\ref{thm:PDA:pro:rec}
to compute  the mean and variance of $\pf_n$, as well as the covariance and correlation between $\pf_n$ and $\ch_n$. Note that they are negatively correlated.

\begin{proposition}
\label{prop:pda:pf}
For $n \geq 6$, we have
\begin{eqnarray}
	\ee _{\PDA}(\pf_{n}) &=& \frac{n (n - 1) (n - 2)}{2 (2n - 5) (2n - 7)} = \frac{n}{8} + \frac{3}{8} + O(n^{-1}), \label{eq:pda_pitchfork_mean} \\
    \cov _{\PDA}(\pf_{n}, \ch_{n}) &=& - \frac{3 n (n - 1) (n - 2) (n - 5)}{2 (2n - 5)^2 (2n - 7) (2n - 9)} = - \frac{3}{32} - \frac{15}{32n} + O(n^{-2}), \label{eq:pda_covariance} \\
    \var _{\PDA}(\pf_{n}) &=& \frac {3 n (n - 1) (n - 2) (4n^4 - 76n^3 + 527n^2 - 1555n + 1610)}{4 (2n - 5)^2 (2n - 7)^2 (2n - 9) (2n - 11)}
    = \frac{3n}{64} - \frac{45}{128n} + O(n^{-2}).\label{eq:pda_pitchfork_variance}
\end{eqnarray}
\noindent
Moreover, $\rho_{\PDA}(\pf_6, \ch_6)=\rho_{\PDA}(\pf_7, \ch_7)=-1$, and the correlation coefficient sequence $\{\rho_{\PDA}(\pf_n, \ch_n)\}_{n\ge 7}$ increases monotonically in $n$ towards zero.
\end{proposition}

\begin{proof}
The Proposition holds for $n=6$ because by~\eqref{eq:six:distribution} we have
\begin{equation}
\label{pf:prop:pda:pf:initial}
\ee _{\PDA}(\pf_{6})=12/7,  \quad   \cov _{\PDA}(\pf_{6}, \ch_{6}) =-12/49,
\quad~~\mbox{and}~~\quad
\var _{\PDA}(\pf_{6})=24/49.
\end{equation}
As the remainder proof is similar to that of Corollary~\ref{cor:pda:ch}, we only outline the main differences here.

First, substituting $\varphi(x, y) = x$ into Theorem \ref{thm:pda_arbitrary_function} and using \rev{(\ref{eq:pda_cherry_mean})} shows that for $n\ge 6$ we have
\begin{equation*}
	\ee _{\PDA}(\pf_{n+1}) = \frac{2n - 7}{2n - 3} \ee _{\PDA}(\pf_{n}) + \frac{3}{2n - 3} \ee _{\PDA}(\ch_{n})
	=\frac{2n - 7}{2n - 3} \ee _{\PDA}(\pf_{n})+\frac{3n(n-1)}{2(2n-3)(2n-5)}.
\end{equation*}
Solving this recurrence with the summation factor $(2n - 3) (2n - 5)$ and \twu{the first equality in~\eqref{pf:prop:pda:pf:initial}} yields~(\ref{eq:pda_pitchfork_mean}).

Next, substituting $\varphi(x, y) = xy$ into Theorem \ref{thm:pda_arbitrary_function} shows that for $n\ge 6$, we have
\begin{eqnarray*}
	\ee _{\PDA}(\pf_{n+1} \ch_{n+1}) &=& \frac{2n - 9}{2n - 3} \ee _{\PDA}(\pf_{n} \ch_{n}) + \frac{n - 1}{2n - 3} \ee _{\PDA}(\pf_{n}) + \frac{3}{2n - 3} \ee _{\PDA}(\ch_{n}^2) \\
   &=&	\frac{2n - 9}{2n - 3} \ee _{\PDA}(\pf_{n} \ch_{n}) + \twu{\frac{n(n-1)(5n^2-9n-20)}{4(2n - 3)(2n-5)(2n-7)}},
\end{eqnarray*}
where the second equality follows from~\eqref{eq:pda_cherry_square} and~(\ref{eq:pda_pitchfork_mean}). Solving this recurrence with the summation factor $(2n - 3) (2n - 5) (2n - 7)$ and \twu{the initial condition $\ee _{\PDA}(\pf_{6} \ch_{6}) = 24/7$} leads to
\begin{equation}
\label{eq:pda_ch_pf}
	\ee _{\PDA}(\pf_{n} \ch_{n}) = \frac{1}{4} \frac{n (n^4 - 6n^3 + 5n^2 + 12n - 12)}{(2n - 5) (2n - 7) (2n - 9)},
\end{equation}
\noindent
from which~(\ref{eq:pda_covariance}) follows.

Finally, substituting $\varphi(x, y) = x^2$ into Theorem \ref{thm:pda_arbitrary_function} shows that for $n\ge 6$, we have
\begin{equation*}
	\ee _{\PDA} (\pf_{n+1}^2) = \frac{2n - 11}{2n - 3} \ee _{\PDA} (\pf_{n}^2) + \frac{6}{2n - 3} \ee _{\PDA}(\pf_{n} \ch_{n}) - \frac{2}{2n - 3} \ee _{\PDA} (\pf_{n}) + \frac{3}{2n - 3} \ee _{\PDA} (\ch_{n}).
\end{equation*}
By~(\ref{eq:pda_cherry_mean}),~(\ref{eq:pda_pitchfork_mean}), and~(\ref{eq:pda_ch_pf}), this recurrence can be solved using  the summation factor $(2n - 3) (2n - 5) (2n - 7) (2n - 9)$ and the initial condition $\ee _{\PDA} (\pf_{6}^2) = 24/7$ to give
\begin{equation*}
	\ee _{\PDA} (\pf_{n}^2) = \frac{1}{4} \frac {n (n^5 - 7n^4 - 19n^3 + 229n^2 - 480n + 276)}{(2n - 5) (2n - 7) (2n - 9) (2n - 11)},
	\end{equation*}
\noindent
from which~(\ref{eq:pda_pitchfork_variance}) follows.

\rev{
Since $\rho_\PDA(\pf_n,\ch_n)=\cov_\PDA(\pf_n,\ch_n)/\sqrt{\var_\PDA(\pf_n)\var_\PDA(\ch_n)}$, it is clear that $\rho_{\PDA}(\pf_6, \ch_6)=\rho_{\PDA}(\pf_7, \ch_7)=-1$ and  the sequence $\{\rho_\PDA(\pf_n,\ch_n)\}_{n\ge 7}$ converges to $0$. Hence it remains to show that this sequence is decreasing. To this end, let $g(n)=4n^4-76n^3+527n^2-1555n+1610$. Then  it suffices to show that the ratio
$$
R(n)=\frac{\rho^2_u(\pf_n,\ch_n)}{\rho^2_u(\pf_{n+1},\ch_{n+1})}=\frac{(n-2)(n-3)(n-5)(2n-11)(2n-7)^2g(n+1)}{(n-1)(n-4)^2(2n-5)(2n-9)^2g(n)}
$$
is greater than 1 for $n\ge 7$. Since $R(n)>1$ holds for $7\le n \le 15$ by numerical computation, we may further assume $n\ge 16$.  Denote the numerator and denominator of $R(n)$ by $R_1(n)$ and $R_2(n)$, respectively. Then using $n\ge 16$ we have
{\small
\begin{eqnarray*}
R_1(n)-R_2(n)&=&(n-6)[(8n-109)^2n^6+(8444n-132200)\twu{n^5}+(519316n-1250941)n^3+(1775858n-1318095)n+364350] \\
&>&0,
\end{eqnarray*}
}
which completes the proof.
}
\epf
\end{proof}

\section{Subtree distributions under the YHK model}
\label{sec:results_yhk}

\rev{In this section, we study the joint distribution of the random variables $A_n$ (i.e. the number of pitchforks) and $B_n$ (i.e. the number of cherries) under the YHK model.  For easy comparison, the results are presented in an order similar to that in Section~\ref{sec:results_pda}. }

\rev{Our starting point is the following recursion on \rev{the} joint distribution, whose proof is omitted.
}

\begin{theorem} \label{thm:yhk_pitchfork_probability}
\rev{Let $n\ge 6$. Then we have $\pyule(\pf_{n}=a, \ch_{n}=b)=0$ if either $a\in \{-1,n,n+1\}$ or $b\in \{1,n\}$ holds. Furthermore, for $0\le a \le n$ and $1< b \le n$ we have}
\begin{align}
    \pyule&(\pf_{n+1} = a, \ch_{n+1} = b) =
    \frac{2a}{n}\,\pyule(\pf_{n} = a, \ch_{n} = b) + \frac{a + 1}{n}\,\pyule(\pf_{n} = a + 1, \ch_n = b - 1) \nonumber \\
    &\quad \quad \quad+ \frac{2(b - a + 1)}{n}\,\pyule(\pf_{n} = a - 1, \ch_{n} = b) +
    \frac{n - a - 2b + 2}{n}\,\pyule(\pf_{n} = a, \ch_{n} = b - 1).
\end{align}

\end{theorem}

As described in Section~\ref{subsection:model}, the tree generating schemes of both models are similar, with the main difference being that the YHK model uses a random pendant edge at each step while the PDA model uses a random edge.  As a result, the proof of Theorem~\ref{thm:yhk_pitchfork_probability} is similar to that of Theorem~\ref{thm:PDA:pro:rec}; the main differences are the expressions in~\eqref{eq:case1}-\eqref{eq:case4}, where certain terms in the numerators become zero and \twu{the denominator $2n-3$ (the total number of edges in $T_n$) is replaced with $n$ (the total number of pendant edges in $T_n$)}.
  Note also that the coefficients in the above recursion are exactly the same as in the rooted case~\citep{WuChoi16}, but the initial values are different.


To study the moments of $\pf_n$ and $\ch_n$, we  present below a functional recursion form of Theorem~\ref{thm:yhk_pitchfork_probability}, whose proof is similar to that of Theorem~\ref{thm:pda_arbitrary_function} and hence omitted here.

\begin{theorem} \label{thm:yhk_arbitrary_function}

Let $\varphi: \mathbb{N} \times \mathbb{N} \to \mathbb{R}$ be an arbitrary function. Then for $n\ge 6$ we have
\begin{align}
    \ee _{\YHK} \varphi(\pf_{n+1}, \ch_{n+1}) &= \frac{2}{n} \ee _{\YHK}[\pf_{n} \: \varphi (\pf_{n}, \ch_{n})] + \frac{1}{n} \ee _{\YHK}[\pf_{n} \: \varphi (\pf_{n}-1, \ch_{n}+1)]  \nonumber \\
    & \quad + \frac{2}{n} \ee _{\YHK}[ (\ch_{n} - \pf_n) \: \varphi (\pf_{n} + 1, \ch_{n})]  + \frac{1}{n}  \ee _{\YHK}[(n-\pf_{n}-2\ch_{n}) \: \varphi (\pf_{n}, \ch_{n}+1)].
\end{align}
\end{theorem}


Theorem~\ref{thm:yhk_arbitrary_function} leads to the following proposition on cherry distributions, whose proof is similar to that in Proposition~\ref{prop:pda_cherry_arbitrary_function} and hence omitted here.

\begin{proposition} \label{prop:yhk_cherry_arbitrary_function}
Let $\psi: \mathbb{N}  \to \mathbb{R} $ be an arbitrary function. Then for $n\ge 4$ we have
\begin{align} \label{eq:ch:yule:functional}
    \ee _{\YHK} &\psi(\ch_{n+1}) = \frac{1}{n} \ee _{\YHK}[ 2 \ch_n \: \psi (\ch_{n}) + (n - 2\ch_{n}) \: \psi (\ch_{n}+1)].
\end{align}
\end{proposition}


Proposition \ref{prop:yhk_cherry_arbitrary_function} enables us to work out the central moments of the cherry distribution for the YHK model.

\begin{corollary} \label{cor:yhk_cherry_distribution}
We have $\var _{\YHK}(\ch_4)=0$,
\begin{equation} \label{eq:yhk_cherry_mean}
	\ee _{\YHK}(\ch_n)=\frac{n}{3}+\frac{4}{(n-1)(n-2)} = \frac{n}{3} + \frac{4}{n^2} + O(n^{-3})~\mbox{for $n\ge 4$}, ~~\mbox{and}~~
\end{equation}
\begin{equation} \label{eq:yhk_cherry_variance}
	\var _{\YHK}(\ch_n)=\frac{2n}{45}-\frac{4(n^2-3n+14)}{3(n-1)^2(n-2)^2} = \frac{2n}{45} - \frac{4}{3n^2} + O(n^{-3})~\mbox{for $n \ge 5$}.
\end{equation}

\end{corollary}

\begin{proof}
By~\eqref{eq:ch:four:five} we have $\ee _{\YHK}(\ch_n)=2$ and $\var _{\YHK}(\ch_n)=0$ for $n\in\{4,5\}$. Hence \eqref{eq:yhk_cherry_mean} holds for $n\in \{4,5\}$ and~\eqref{eq:yhk_cherry_variance} holds for $n=5$.

Substituting $\psi(x)=x$ in~(\ref{eq:ch:yule:functional})  shows that
\begin{equation}
\label{pf:eq:yhk:ch:recur}
    \ee _{\YHK}(\ch_{n+1})=\frac{1}{n}\ee _{\YHK}\big[ 2\ch^2_n+(n-2\ch_n)(\ch_n+1)\big]
    =1+\frac{n-2}{n}\ee _{\YHK}(\ch_n)~~\mbox{for $n\ge 4$}.
\end{equation}
\noindent
This recurrence can be solved by the summation factor method. 
First, we multiply both sides of~\eqref{pf:eq:yhk:ch:recur} by \twu{the} summation factor
$s_n=(n-1)^{\rf{2}}$. Let $f(n) = s_{n-1}\ee_{\YHK}(\ch_n)=(n-2)^{\rf{2}} \, \ee_{\YHK}(\ch_n)$. Substituting $n$ with $n - 1$ in~\eqref{pf:eq:yhk:ch:recur} shows that for $n\ge 5$ we have
\begin{equation*}
	f(n) = (n-2)^{\rf{2}} + f(n-1)=f(4) + \sum_{k=3}^{n-2} k^{\rf{2}} = 4 + \sum_{k=1}^{n-2} k^{\rf{2}} = 4 + \frac{n(n - 1)(n - 2)}{3},
\end{equation*}
from which~(\ref{eq:yhk_cherry_mean}) follows. Here the third equality follows from $f(4)=12$ and the fourth equality follows from~(\ref{eq:telescopic_sum}).

\rev{Applying Proposition~\ref{prop:yhk_cherry_arbitrary_function} with $\psi(x)=x^2$ and using~(\ref{eq:yhk_cherry_mean})}, we have
\begin{equation}
\label{pf:eq:yhk:ch:square:recur}
	\ee _{\YHK}(\ch_{n+1}^2) = \frac{n - 4}{n} \ee _{\YHK} (\ch_n^2)+ \frac{2n + 1}{3} + \frac{8}{n (n - 2)} \quad~\mbox{\twu{for} $n \ge 4$.}
\end{equation}
Consider \twu{the} summation factor $s^*_n=(n-3)^{\rf{4}}$ and let $f^*(n):= (n-4)^{\rf{4}} \ee_{\YHK} (B_n^2)$. Using~\eqref{pf:eq:yhk:ch:square:recur}, by $f^*(4)=0$ and~(\ref{eq:telescopic_sum}) we have
\begin{eqnarray*}
    f^*(n)
    &=& f^*(n-1)+\frac{(2n-1)(n-4)^{(4)}}{3}+8(n-2)(n-4)\\
    &=& \frac{1}{3} \sum_{k=5}^n (2k-1) (k-4)^{\rf{4}}  + 8 \sum_{k=5}^n (k-2)(k-4) \\
    &=& \frac{2}{3} \sum_{k=5}^n (k-4)^{\rf{5}} - \frac{1}{3} \sum_{k=5}^n (k-4)^{\rf{4}} + 8 \sum_{k=5}^n (k-3)^{\rf{2}} - 8 \sum_{k=5}^n (k-2) \\
    &=& \frac{(5n+2) (n-4)^{\rf{5}} }{45}  + \frac{8 (n-3)^{\rf{3}}}{3}   - 4 (n-2)^{\rf{2}} +8
\end{eqnarray*}
for $n\ge 5$.  This implies
\begin{equation}
\label{eq:yhk:mean:ch:square}
    \ee_{\YHK}(B_n^2)
    = \frac{n(5n+2)}{45}+\frac{4(2n-1)}{3(n-1)(n-2)}~\mbox{for $n\ge 5$},
\end{equation}
from which~(\ref{eq:yhk_cherry_variance}) follows.
\epf
\end{proof}



\rev{
The following theorem also follows directly from Proposition~\ref{prop:yhk_cherry_arbitrary_function}, whose proof is omitted as it is similar to that of Theorem~\ref{thm:pda_cherry_probability_distribution}. However, Theorem~\ref{thm:pda_cherry_probability_distribution} \twu{provides} a closed-form formula in~\eqref{eq:pda_cherry_pd} for the distribution of $\ch_n$ under the PDA model; \twu{whether} such a closed-form formula exists under the YHK model remains to be seen.
}

\begin{theorem}
 \label{thm:yule_cherry_probability_distribution}
For $n\ge 4$ and $1<k<n$ we have
\begin{equation} \label{eq:yhk_cherry_probability}
    \pp _{\YHK} (\ch_{n+1} = k) = \frac{2k}{n} \pp _{\YHK} (\ch_{n} = k) + \frac{n - 2k + 2}{n}\pp _{\YHK} (\ch_{n} = k - 1).
\end{equation}
\end{theorem}

We complete this section by using Theorem~\ref{thm:yhk_arbitrary_function}
to compute  the mean and variance of $\pf_n$, as well as the covariance and correlation between $\pf_n$ and $\ch_n$, under the YHK model.

\begin{proposition}
\label{prop:yhk:pf}
We have 
\twu{$\var _{\YHK}(\pf_{6})=16/25,$} 
\begin{eqnarray}
    \ee _{\YHK}(\pf_n) &=& \frac{n}{6} + \frac{4(2n - 3)}{(n - 1)(n - 2)(n - 3)} \,= \frac{n}{6} + \frac{8}{n^2} + O(n^{-3})~~\mbox{for $n\ge 6$}, \label{eq:yhk_pitchfork_mean} \\
    \cov_{\YHK}(\pf_n,\ch_n) &=&
    -\frac{n}{45} - \frac{4(n^3-6n^2+35n-42)}{3(n-1)^2(n-2)^2(n-3)}\,
    = -\frac{n}{45}-\frac{4}{3n^2} + O(n^{-3})~~\mbox{for $n\ge 6$, and}
    \label{eq:yhk_covariance}\\
    \var _{\YHK}(\pf_n) &=&
    \frac{23n}{420}-\frac{16(2n-3)^2}{(n-1)^2(n-2)^2(n-3)^2}\,
    = \frac{23n}{420}-\frac{64}{n^4} + O(n^{-5})~~\mbox{for $n\ge 7$}.
    \label{eq:yhk_pitchfork_variance}
\end{eqnarray}
Moreover, $\rho_{\YHK}(\pf_6, \ch_6)=\rho_{\YHK}(\pf_7, \ch_7)=-1$, and the correlation coefficient sequence $\{\rho_{\YHK}(\pf_n, \ch_n)\}_{n\ge 7}$ increases monotonically in $n$ towards the constant $-\sqrt{14/69}$.
\end{proposition}

\begin{proof}
The proposition holds for $n=6$ because by~\eqref{eq:six:distribution} we have
\begin{equation}
\label{pf:prop:yule:pf:initial}
\ee _{\YHK}(\pf_{6})=8/5,  \quad   \cov _{\YHK}(\pf_{6}, \ch_{6}) =-8/25,
\quad~~\mbox{and}~~\quad
\var _{\YHK}(\pf_{6})=16/25.
\end{equation}
As the remainder \twu{of the} proof is similar to that of Corollary~\ref{cor:yhk_cherry_distribution}, we only outline the main differences here.

Applying Theorem~\ref{thm:yhk_arbitrary_function} \rev{with} $\varphi(x,y)=x$ and using Corollary~\ref{cor:yhk_cherry_distribution}, we have
\begin{align*}
	\ee _{\YHK}(\pf_{n+1})&= \frac{n-3}{n}\ee _{\YHK}(A_n)+\frac{2}{3}+\frac{8}{n(n-1)(n-2)}~~\mbox{for $n\ge 6$}.
\end{align*}
Using the summation factor $(n - 2)^{\rf{3}}$ and the initial value \twu{$\ee _{\YHK}(A_6)=8/5$} 
in~\eqref{pf:prop:yule:pf:initial}, we can solve the last recursion to obtain~(\ref{eq:yhk_pitchfork_mean}).

Next, applying Theorem~\ref{thm:yhk_arbitrary_function} \rev{with} $\varphi(x, y) = xy$ gives us
\begin{align}
    \ee_{\YHK}(\pf_{n+1}\ch_{n+1})
    &=\frac{n-5}{n} \ee_{\YHK} (\pf_n\ch_n)+\frac{n-1}{n} \ee_{\YHK}(\pf_n)+\frac{2}{n}\ee_{\YHK}(\ch_n^2) \nonumber \\
    &=\frac{n-5}{n} \ee_{\YHK}(\pf_n\ch_n)+\frac{35n-7}{90}+\frac{4(10n^2-29n+15)}{3(n-3)^{\rf{4}}}~~\mbox{for $n\ge 6$.} \label{eq:yhk:mixed:recur}
\end{align}
\noindent
Here the second equality follows from~\eqref{eq:yhk:mean:ch:square} and~\eqref{eq:yhk_pitchfork_mean}. Consider the summation factor  $s_n=(n-4)^{\rf{5}}$ and let $f(n)=s_{n-1} \ee_{\YHK}(\pf_n\ch_n)$. Since $f(6)=384$, by~\eqref{eq:yhk:mixed:recur} and~\eqref{eq:telescopic_sum} we have
\begin{align*}
    f(n) &= f(6) + \sum_{k=2}^{n-5} \left(\frac{7k^{\rf{6}}}{18}-\frac{7k^{\rf{5}}}{15}+\frac{40k^{\rf{3}}}{3}+28k^{\rf{2}}+24k^{\rf{1}} \right) \\
    &=\frac{(n-5)^{\rf{7}}}{18}-\frac{7(n-5)^{\rf{6}}}{90}+\frac{10(n-5)^{\rf{4}}}{3}+\frac{28(n-5)^{\rf{3}}}{3}+12(n-5)^{\rf{2}}
\end{align*}
for $n \geq 7$. This gives us
\begin{equation} \label{eq:yhk_cherry_pitchfork_product_mean}
	\ee_{\YHK}(\pf_n\ch_n) = {\frac {5n^{4} - 27n^{3} + 40n^{2} + 288n - 360}{90 (n - 3) ( n-2)}}\,~\mbox{for $n \geq 7$,}
\end{equation}
\noindent
from which~(\ref{eq:yhk_covariance}) follows. 

To obtain $\var _{\YHK}(\pf_n)$, applying Theorem~\ref{thm:yhk_arbitrary_function} \rev{with} $\varphi(x,y)=x^2$ we have
\begin{align}
	\ee_{\YHK}(\pf_{n+1}^2) &= \frac{n-6}{n} \ee_{\YHK} (\pf_n^2) - \frac{1}{n} \ee_{\YHK}(\pf_n) + \frac{2}{n}\ee_{\YHK}(\ch_n) + \frac{4}{n} \ee_{\YHK}(\pf_n \ch_n) \nonumber \\
    &= \frac{n-6}{n} \ee_{\YHK} (\pf_n^2) + \frac {20n^{5} - 83n^{4} - 2n^{3} + 1487n^{2} - 2862n + 360}{90n (n - 1) (n - 2) (n - 3) }~~\mbox{for $n\ge 6$}, \label{eq:yhk:pf:var:rec}
    \end{align}
\noindent
where the second \rev{equality} follows from~(\ref{eq:yhk_cherry_mean}), (\ref{eq:yhk_pitchfork_mean}),  and (\ref{eq:yhk_cherry_pitchfork_product_mean}). Using the  summation factor $(n - 6)^{\rf{6}}$ and the initial condition $\ee_{\YHK}(\pf_{6}^2) = 16/5$, we solve the recurrence in~\eqref{eq:yhk:pf:var:rec} to show that
\begin{equation*}
	\ee_{\YHK} (\pf_n^2) = \frac{n (35n^4 - 141n^3 - 29n^2 + 3909n - 5454)}{1260 (n - 3) (n - 2) (n - 1)}~\mbox{for $n\ge 7$},
\end{equation*}
\noindent
from which~(\ref{eq:yhk_pitchfork_variance}) follows. 

\twu{
To establish the last statement in the proposition, note that $\rho_{\YHK}(\pf_6, \ch_6)=\rho_{\YHK}(\pf_7, \ch_7)=-1$ clearly holds.} For simplicity let
\begin{eqnarray*}
f_c(n) &=& n^6 - 9n^5 + 31n^4 + 9n^3 - 320n^2 + 2088n - 2520, \\
g_a(n) &=& 23n(n-1)^2(n-2)^2(n-3)^2-420\times16(2n-3)^2,~\mbox{and}~ \\
g_b(n) &=& 2n(n-1)^2(n-2)^2-60(n^2-3n+14).
\end{eqnarray*}
Then we have
$$
\rho_\YHK(\pf_n,\ch_n)=\frac{\cov_\YHK(\pf_n,\ch_n)}{\sqrt{\var_\YHK(\pf_n)\var_\YHK(\ch_n)}}
=\frac{-\sqrt{84}f_c(n)}{3\sqrt{g_a(n)g_b(n)}},
$$
by which it follows that $\rho_\YHK(\pf_n,\ch_n)$ converges to $-\sqrt{14/69}$. Hence it remains to show that this sequence is decreasing. To this end it suffices to show that
$$
R(n)=\frac{\rho^2_\YHK(\pf_n,\ch_n)}{\rho^2_\YHK(\pf_{n+1},\ch_{n+1})}=\frac{f^2_c(n)g_a(n+1)g_b(n+1)}{f^2_c(n+1)g_a(n)g_b(n)}
$$
is greater than 1 for $n\ge 7$.  Since $R(n)>1$ holds for $7\le n \le 19$ by numerical computation, we may further assume $n\ge 20$.  Denote the numerator and denominator of $R(n)$ by $R_1(n)$ and $R_2(n)$, respectively. Then for $n\ge 20$ we have
{\small
\begin{eqnarray*}
R_1(n)-R_2(n)&=&
 60(n-1)^2[(345n - 6900)n^{17} + (78310n - 567730)n^{15} + (2579592n^2 - 5970866n - 6240646)n^{12}  \\
 && \quad + (119332554n - 613621681)n^{10} + (2240461678n - 7810653888)n^8 + (26499724352n - 84099771728)n^6  \\
 && \quad + (197227522272n - 276252591744)n^4 + 190463028480n^3 + (59732121600n - 201790310400)n + 73156608000]\\
 &>&345(n -20)n^{17} + 10(7831n - 56773)n^{15} + (2579592n^2 - 5970866n - 6240646)n^{12}  \\
 && \quad + (119332554n - 613621681)n^{10} + (2240461678n - 7810653888)n^8 + (26499724352n - 84099771728)n^6  \\
 && \quad + (197227522272n - 276252591744)n^4 + 12441600(4801n - 16219)n \\
&>&0. 
\end{eqnarray*}
}
\epf
\end{proof}

\section{Comparison of random models}
\label{sec:comparison}

In this section, we compare and contrast the properties of the cherry and pitchfork distributions under the PDA and the YHK models, \rev{in} both rooted and unrooted cases. To this end, for $n \geq 4$, let $\pf^*_n$ and $\ch^*_n$ be the random variables $\pf(T^*)$ and $\ch(T^*)$, respectively, for a random tree $T^*$ in $\rtsp_n$.

\subsection{Mean and Variance}

\rev{In} Table~\ref{tb:mv} we collect the means, variances, and covariances for pitchfork and cherry distributions under these two models, \rev{for both} rooted and unrooted trees. The entries for unrooted trees follow from Sections~\ref{sec:results_pda} and~\ref{sec:results_yhk}, while those for rooted trees follow from the relevant results in~\citet[Sections 3 and 4]{WuChoi16}.


 \begin{table}[h]
       \begin{tabular}{ | c | c | c| }
    \hline
     & PDA & YHK \\ \hline
  $\ee(\pf^*_n)$ & $\frac{n(n-1)(n-2)}{2(2n-3)(2n-5)}$ & $\frac{n}{6}$ \\ \hline
$\ee(\pf_n)$   & $\frac{n(n-1)(n-2)}{2(2n-5)(2n-7)}$ & $\frac{n}{6}+\frac{4(2n-3)}{(n-1)(n-2)(n-3)}$ \\    \hline
  $\ee(\ch^*_n)$ & $\frac{n(n-1)}{2(2n-3)}$ & $\frac{n}{3}$ \\ \hline
$\ee(\ch_n)$   & $\frac{n(n-1)}{2(2n-5)}$ & $\frac{n}{3}+\frac{4}{(n-1)(n-2)}$ \\
\hhline{|===|}
      $\var(\pf^*_n)$ &
  $\frac{3n(n-1)(n-2)(n-3)(4n^3-40n^2+123n-110)}{4(2n-3)^2(2n-5)^2(2n-7)(2n-9)}$
  & $\frac{23n}{420}$~ {\tiny ($n\ge 7$) } \\ \hline
$\var(\pf_n)$   &
  $\frac{3n(n-1)(n-2)(4n^4-76n^3+527n^2-1555n+1610)}{4(2n-5)^2(2n-7)^2(2n-9)(2n-11)}$
 &   $\frac{23n}{420}-\frac{16(2n-3)^2}{(n-1)^2(n-2)^2(n-3)^2}$~ {\tiny ($n\ge 7$) } \\    \hline
  $\var(\ch^*_n)$ & $\frac{n(n-1)(n-2)(n-3)}{2(2n-3)^2(2n-5)}$
  & $\frac{2n}{45}$~ {\tiny ($n\ge 5$)} \\ \hline
$\var(\ch_n)$   & $\frac{n(n-1)(n-4)(n-5)}{2(2n-5)^2(2n-7)}$ & $\frac{2n}{45} - \frac{4(n^2-3n+14)}{3(n-1)^2(n-2)^2}$ ~ {\tiny ($n\ge 5$)} \\
\hhline{|===|}
 $\cov(\pf^*_n,\ch^*_n)$ & $-\frac{n(n-1)(n-2)(n-3)}{2(2n-3)^2(2n-5)(2n-7)}$ &
 $-\frac{n}{45}$~ {\tiny ($n\ge 6$)} \\
\hline
$\cov(\pf_n,\ch_n)$ & $-\frac{3 n (n - 1) (n - 2) (n - 5)}{2 (2n - 5)^2 (2n - 7) (2n - 9)}$ &
$-\frac{n}{45} - \frac{4(n^3-6n^2+35n-42)}{3(n-1)^2(n-2)^2(n-3)}$ \\
\hline
  \end{tabular}
  \caption{The means, variances, and covariances of pitchfork distributions and cherry distributions for the YHK and PDA models. Note that other than the five formulas  where the range of $n$ is explicitly given in the brakets, we have $n\ge 6$ for the formulas on $\ee(\pf_n)$, $\var(\pf_n)$ and $\cov(\pf_n,\ch_n)$, and $n\ge 4$ for all the others.
  }
\label{tb:mv}
\end{table}

Using the entries in \rev{Table~\ref{tb:mv}}, we show below that a tree generated under the YHK model typically contains more cherries and more pitchforks than  \rev{one} under the PDA model, which is in line with
\rev{the numerical results seen in Fig.~\ref{fig:probability_contours}}.

 \begin{proposition}
{\rm (i):} For $n\ge 6$ and $Y_n\in \{\pf^*_n,\ch^*_n,\ch_n\}$, we have
\begin{equation}
\ee_\PDA(Y_n)< \ee_\YHK(Y_n) < \frac{4}{3}\ee_\PDA(Y_n).
\label{eq:yu:compare}
\end{equation}
Furthermore, we have $\ee_\PDA(\pf_n)< \ee_\YHK(\pf_n)$ for $n\ge 12$, and
$\ee_\YHK(\pf_n) < \frac{4}{3}\ee_\PDA(\pf_n)$ for $n\ge 6$.  \\
{\rm (ii):} As $n \to \infty$, we have
\begin{equation}
	\ee (\pf_n) \sim \frac{1}{2} \ee (\ch_n)\quad \mbox{and}\quad \ee (\pf^*_n) \sim \frac{1}{2} \ee (\ch^*_n)
	\label{ratio:pf:ch}
\end{equation}
\noindent
 under both the YHK and the PDA models.
 \end{proposition}

 \begin{proof}
{\rm (i) } When $Y\in \{\pf^*_n,\ch^*_n\}$, \twu{the inequalities in~(\ref{eq:yu:compare})} follow from~\citet[Proposition 6]{WuChoi16}.

 Next,  \twu{the inequalities in~(\ref{eq:yu:compare})}  hold for the unrooted cherry distributions because we have $\ee_\PDA(\ch_6)=\frac{15}{7}<\frac{11}{5}=\ee_\YHK(\ch_6)$,
 $$
 \ee_\PDA(\ch_n)=\frac{n(n-1)}{2(2n-5)} \le \frac{n}{3} < \ee_\YHK(\ch_n)
~\mbox{for $n\ge 7$,}
 $$
 and
 $$
  \ee_\YHK(\ch_n)< \frac{n}{3}+\frac{1}{3}=\frac{n+1}{3} <\frac{2n(n-1)}{3(2n-5)}
  =\frac{4}{3}\ee_\PDA(\ch_n)~\mbox{for $n\ge 6$}.
  $$

For the unrooted pitchfork distributions, we have
 $$
  \ee_\YHK(\pf_n)-\ee_\PDA(\pf_n)=\frac{4(2n-3)}{(n-1)(n-2)(n-3)}+\frac{n[(n-7)(n-8)-27]}{6(2n-5)(2n-7)}>0
 $$
for $n\ge 12$ (but note that $\ee_\YHK(\pf_{11})<\ee_\PDA(\pf_{11})$). Moreover, it is straightforward to check that $\ee_\YHK(\pf_n) < 4\ee_\PDA(\pf_n)/3$ holds for
$n\in \{6,7\}$, and we also have
 $$
   \ee_\YHK(\pf_n) < \frac{n+2}{6} <\frac{2n(n-1)(n-2)}{3(2n-5)(2n-7)}
     =\frac{4}{3}\ee_\PDA(\pf_n)
     \quad \mbox{for $n\ge 8$.}
 $$

 {\rm (ii) }  \rev{Note that}~(\ref{ratio:pf:ch}) holds for the YHK model
 \rev{because
 $\ee_\YHK(\pf^*_n)=\ee_\YHK(\ch^*_n)/2$ for} $n\ge 4$ and
\twu{
 $$
 \ee_\YHK(\pf_n) \sim  \ee_\YHK(\pf^*_n)
~\quad~\mbox{and}~\quad~
 \ee_\YHK(\ch_n) \sim   \ee_\YHK(\ch^*_n).
 $$
}
 \rev{Furthermore}, we have~(\ref{ratio:pf:ch}) for the PDA model because
  $$
 \lim_{n\to \infty}\frac{\ee_\PDA(\pf_n)}{\ee_\PDA(\ch_n)}
=\frac{1}{2},~\quad~
\twu{
 \ee_\PDA(\pf_n) \sim   \ee_\PDA(\pf^*_n),
~\quad~\mbox{and}~\quad~
 \ee_\PDA(\ch_n) \sim   \ee_\PDA(\ch^*_n).  
}
 $$
\epf
 \end{proof}

 \medskip
Next, we compare the means of the cherry and pitchfork distributions between rooted trees and unrooted trees. Note that the limiting difference between the two models \rev{is}  different, \twu{which is in line with the result on cherry distributions reported by~\citet[Lemma 6]{McKenzie2000}.}

 \begin{proposition}
 \label{prop:mean:comp}
For $n\ge 6$ and $Y_n\in \{\pf_n,\ch_n\}$, we have
\begin{equation}
\ee_\YHK(Y_n^*)< \ee_\YHK(Y_n)\le \ee_\YHK(Y_n^*)+\frac{3}{5}~\quad~\mbox{and}~\quad~
\ee_\PDA(Y_n^*)+\frac{1}{4}< \ee_\PDA(Y_n)\le \ee_\PDA(Y_n^*)+\frac{16}{21}.
\end{equation}
Moreover, $\{\ee(Y_n)-\ee(Y_n^*)\}_{n\ge 6}$ is a  strictly decreasing sequence whose limit is $0$ under the YHK model, and $1/4$ under the PDA model.
\end{proposition}

\begin{proof}
Define  $a_n=\ee_\YHK(A_n)-\ee_\YHK(A_n^*), b_n=\ee_\YHK(B_n)-\ee_\YHK(B_n^*), \tilde{a}_n=\ee_\PDA(A_n)-\ee_\PDA(A_n^*)$ and $ \tilde{b}_n=\ee_\PDA(B_n)-\ee_\PDA(B_n^*) $. Then we have
$$
a_n =\frac{4(2n-3)}{(n-1)(n-2)(n-3)}
~\quad~
\mbox{and}
~\quad~
b_n =\frac{4}{(n-1)(n-2)}.
$$
Thus we have $a_n, b_n >0$ and $\lim_{n\to \infty} a_n =\lim_{n\to \infty} b_n=0$. Moreover, it is straightforward to check that$a_{n+1}/a_n<1$ and 
$b_{n+1}/b_n<1$ hold for $n\ge 6$, from which we know that $\{a_n\}_{n\ge 6}$ and $\{b_n\}_{n\ge 6}$ are  strictly decreasing sequences whose limits are $0$. This completes the proof of the proposition for the YHK model in view of $a_6=3/5$ and $b_6=1/5$.

Similarly, 
$$
\tilde{a}_n =\frac{2n(n-1)(n-2)}{(2n-3)(2n-5)(2n-7)}
~\quad~
\mbox{and}
~\quad~
\tilde{b}_n =\frac{n(n-1)}{(2n-3)(2n-5)}.
$$
Thus we have  $\tilde{a}_n, \tilde{b}_n >0$ and $\lim_{n\to \infty} {\tilde{a}_n = \lim_{n\to \infty} \tilde{b}_n} =1/4$. Since $\tilde{a}_{n+1}/\tilde{a}_n <1$ and $\tilde{b}_{n+1}/\tilde{b}_n <1$ hold for $n\ge 6$, it follows that 
$\{\tilde{a}_n\}_{n\ge 6}$ and $\{\tilde{b}_n\}_{n\ge 6}$ are  strictly decreasing sequences whose limits are $1/4$. This completes the proof of the proposition for the PDA model in view of $\tilde{a}_6=16/21$ and $\tilde{b}_6=10/21$.
\epf
\end{proof}


In the following proposition, we show that the variance of the number of cherries is greater in the rooted tree case than \twu{it is} in the unrooted tree case regardless of whether \twu{the  YHK model or the PDA model is used}. The same conclusion applies to the variance of the number of pitchforks, \rev{as well as the covariance.}

\begin{proposition}
 \label{prop:var:comp}
For $n\ge 7$ and $Y_n\in \{\pf_n,\ch_n\}$, we have
\begin{equation}
\label{eq:var:comparison}
\var_\YHK(Y_n^*) > \var_\YHK(Y_n)
~\quad~\mbox{and}~\quad~
\var_\PDA(Y_n^*)> \var_\PDA(Y_n).
\end{equation}
Moreover, 
when $n\ge 6$ we have
\begin{equation}
\label{eq:covar:comparison}
\cov_\YHK(\pf^*_n,\ch_n^*)> \cov_\YHK(\pf_n,\ch_n)
~\quad~\mbox{and}~\quad~
\cov_\PDA(\pf^*_n,\ch_n^*)> \cov_\PDA(\pf_n,\ch_n).
\end{equation}
 \end{proposition}

 \begin{proof}
Let $a_n=\var_\YHK(A^*_n)-\var_\YHK(A_n), b_n=\var_\YHK(B^*_n)-\var_\YHK(B_n), \tilde{a}_n=\var_\PDA(A^*_n)-\var_\PDA(A_n)$ and $\tilde{b}_n= \var_\PDA(B^*_n)-\var_\PDA(B_n)$. Then for $n\ge 7$ we have $a_n, b_n >0$ because using Table~\ref{tb:mv} we have
$$
a_n =\frac{16(4n^2-12n+9)}{(n-1)^2(n-2)^2(n-3)^2}>0
~\quad~
\mbox{and}
~\quad~
b_n =\frac{4(n^2-3n+14)}{3(n-1)^2(n-2)^2}>0.
$$

Similarly,  by  Table~\ref{tb:mv}  we have
$$
\tilde{b}_n =\frac{n(n-1)[(2n-9)n^2+n+15]}{(2n-3)^2(2n-5)^2(2n-7)} > 0.
$$
Moreover, we have
$$
\tilde{a}_n = \frac{6n(n-1)(n-2)\kappa(n)}{(2n-3)^2(2n-5)(2n-7)^2(2n-9)(2n-11)},
$$
where
$$
\kappa(n)=6n^3-75n^2+277n-273=n(2n-15)(3n-15)+52n-273.
$$
Then $\tilde{a}_n >0$ holds for $n\ge 7$ because $\kappa(n)>0$ clearly holds for $n\ge 8$ and $\kappa(7)=49>0$.

Finally, when $n\ge 6$, using Table~\ref{tb:mv} we have
\begin{align*}
\cov_\YHK(\pf^*_n,\ch^*_n)-\cov_\YHK(\pf_n,\ch_n)&=\frac{4[n(n-2)(n-4)+27n-42]}{3(n-1)^2(n-2)^2(n-3)}>0\,~\mbox{and} \\
\cov_\PDA(\pf^*_n,\ch_n^*)- \cov_\PDA(\pf_n,\ch_n)
&=\frac{n^2(n-1)(n-2)\,[4(n-7)+39)]}{(2n-3)^2(2n-5)^2(2n-7)(2n-9)}>0.
\end{align*}
 \epf
 \end{proof}

 \rev{Note that for the pitchfork distributions, the condition $n\ge 7$ is optimal for~\eqref{eq:var:comparison} in view of $\var_\YHK(\pf_6^*)=\frac{2}{5} < \frac{16}{25}=\var_\YHK(\pf_6)$ and  $\var_\PDA(\pf_6^*) =\frac{104}{441} < \frac{24}{49}= \var_\PDA(\pf_6)$. However, for the cherry distributions, the proof of Proposition~\ref{prop:var:comp} can be extended to show that both $\var_\YHK(\ch_n^*) > \var_\YHK(\pf_n)$ and  $\var_\PDA(\ch_n^*) > \var_\PDA(\ch_n)$ hold for $n\ge 4$. }

\subsection{Log-Concavity}

In addition to mean and variance, in this subsection we shall use results in Sections~\ref{sec:results_pda} and~\ref{sec:results_yhk}  to gain more insights into the properties of the subtree distributions, particularly the cherry distributions. To this end, denote the probability mass functions (PMFs) of $\ch^*_n$ and $\ch_n$ under the PDA model by
$\pmf^*_n$ and $\pmf_n$, respectively. Similarly, denote the probability mass functions 
 of $\ch^*_n$ and $\ch_n$ under the YHK model by $\ypmf^*_n$ and $\ypmf_n$, respectively.

First, \rev{by~(\ref{eq:pda_cherry_pd}) in Theorem~\ref{thm:pda_cherry_probability_distribution} it follows} that  when $n\ge 4$ we have
\begin{equation}
\label{eq:pmf}
\pmf_n(k)=\pp_\PDA(\ch_n=k)=\frac{n!(n-2)!(n-4)!2^{n-2k}}{(n-2k)!(2n-4)!k!(k-2)!}
~~\mbox{for $2\le k \le n/2$,}
\end{equation}
 and $\pmf_n(k)=0$ otherwise. Moreover, \citet[Theorem 6]{WuChoi16} implies that when  $n\ge 4$ we have
\begin{equation}
\label{eq:rpmf}
\pmf^*_n(k)=\pp_\PDA(\ch^*_n=k)=\frac{n!(n-1)!(n-2)!2^{n-2k}}{(n-2k)!(2n-2)!k!(k-1)!}
~~\mbox{for $1\le k \le n/2$,}
\end{equation}
and $\pmf^*_n(k)=0$ otherwise.

Next, we show that  both $\ypmf_n$ and $\pmf_n$ (i.e. the probability mass functions  
of $B_n$ under the YHK  model and the PDA model, respectively) are log-concave, \rev{and hence also unimodal}. 
To this end,  recall that $\nabla(n)=\lceil n/4 \rceil$. 
\twu{Furthermore, by~\citet[Lemma 3]{WuChoi16} it follows that for four non-negative numbers $z_1,z_2,z_3,z_4$ with $z_2^2\ge z_1z_3$ and $z_3^2\ge z_2z_4$, we have
	\begin{equation} \label{eq:concave:four-number}
	z_2z_3\ge z_1z_4~~\mbox{and}~~z_1z_3+z_2z_4 \ge 2z_1z_4.
	\end{equation}
}

\begin{theorem}
\label{them:log-concav}
For $n \geq 4$ and $2 \le k \le n/2$ we have
\begin{equation} \label{eq:cherry_log_concavity}
\ypmf^2_n(k)> \ypmf_n(k-1)\ypmf_n(k+1)~	
	~\quad~
	\mbox{and}
	~\quad~
	\pmf^2_n(k)> \pmf_n(k-1)\pmf_n(k+1).
\end{equation}
Moreover, when $n> 8$, \rev{then}  $\pmf_n(k-1) < \pmf_n(k)$ for $2 \le k < \nabla(n)$, and $\pmf_n(k) > \pmf_n(k+1)$  for $\nabla(n)\le k \le n/2$.
\end{theorem}

\begin{proof}
 For $n\ge 4$ and $2\leq t \leq n/2$, let
\begin{align}
g_n(t):=\frac{\pmf_n(t-1)}{\pmf_n(t)}
=\frac{4t(t - 2)}{(n - 2t + 1)(n - 2t + 2)}, 
\end{align}
where the equality follows from~\eqref{eq:pmf}. 
Note that $g_n$ is a strictly increasing function, that is, $g_n(t-1)<g_n(t)$ holds for $2< t \le n/2$. Since we have $\pmf_n(k)>0$ if $2\le k \leq n/2$, and $\pmf_n(k)=0$ otherwise, it follows that $\pmf^2_n(k)> \pmf_n(k-1)\pmf_n(k+1)$ holds for $2\le k\le n/2$. This completes the proof of~(\ref{eq:cherry_log_concavity}) for $\pmf_n$.  

Next, assume $n>8$. Since $g_n$ is strictly increasing and $g_n(t)=1$ if and only if $t=(n+1)(n+2)/(2(2n-1))$, by~\eqref{eq:nabla} we have  $g_n(t)<1$ for $2\le t<\nabla(n)$. Together with the fact that $\pmf_n(k)>0$ holds if and only if $2\le k \leq n/2$, this shows that $\pmf_n(k-1) < \pmf_n(k)$ holds for $2\le k< \nabla(n)$.  Similarly, noting that $g_n(t)>1$ holds for  $\nabla(n) < t \le n/2 $,  we have  $\pmf_n(k) > \pmf_n(k+1)$  for $\nabla(n) \le k < n/2 $, from which the last statement of the theorem follows.  

It remains to prove~(\ref{eq:cherry_log_concavity}) for $\ypmf_n$, 
which will be established by induction on $n$. The base case $n=4$ clearly holds. Now assume $n\ge 4$ and let $\beta_t=\ypmf_n(t)$ for $1\le t \le n$. 
Then~(\ref{eq:yhk_cherry_probability}) implies
 \begin{equation}
 \label{eq:concave:proof:ypmf}
 n\ypmf_{n+1}(t)=2t\beta_t+(n-2t+2)\beta_{t-1}~\mbox{for $1< t < n$}. 
 \end{equation}
Furthermore, we have the induction assumption:
 \begin{equation}
\label{eq:concave:proof:ia}
\beta^2_t=\ypmf^2_n(t)> \ypmf_n(t-1)\ypmf_n(t+1)=\beta_{t-1}\beta_{t+1}~\mbox{for $2\le t \le n/2$.}
 \end{equation}
Since 
$\ypmf_{n+1}(t)=0$ holds for $t>(n+1)/2$, we have $\ypmf_{n+1}(k+1)=0$ for $k\ge n/2$. Together with $\ypmf_{n+1}(1)=0$, it follows that for the induction step it suffices to show
\begin{equation}
\label{eq:pdf:lc:induction}
n^2\ypmf^2_{n+1}(k)\,>\,n\ypmf_{n+1}(k-1)n\ypmf_{n+1}(k+1)~\mbox{ for  $2<k<n/2$.}
\end{equation}
To this end, denote the left-hand and right-hand sides of~\eqref{eq:pdf:lc:induction} by $L(k)$ and $R(k)$, respectively, that is, we have $L(k)=n^2\ypmf^2_{n+1}(k)$ and $R(k)=n\ypmf_{n+1}(k-1)n\ypmf_{n+1}(k+1)$.


Now fix a number $k$ with $2<k<n/2$.  Then both
$a=(n-2k+2)^2$ and $b=2k(n-2k+2)+2n$ are greater than $0$.   Furthermore, using~\eqref{eq:concave:proof:ypmf} three times with $t=k$, $t=k-1$, and $t=k+1$, we have
\begin{eqnarray}
	L(k)&=&4k^2\beta_k^2+a\beta^2_{k-1}+4k(n-2k+2)\beta_{k-1}\beta_{k}~~\mbox{and} \label{eq:concave:left}\\
	R(k)&=&(4k^2-4)\beta_{k-1}\beta_{k+1}+(a-4)\beta_{k-2}\beta_{k}+2(n-2k)(k-1)\beta_{k-1}\beta_{k} +(b+8)\beta_{k-2}\beta_{k+1},	\label{eq:concaveright}
\end{eqnarray}
where the equalities $(n-2k)(n-2k+4)=a-4$ and  
$
2(k+1)(n-2k+4)=b+8
$ 
are used to obtain~\eqref{eq:concaveright}.

Since $b=4k(n-2k+2)-2(n-2k)(k-1)$, by~\eqref{eq:concave:left} and~\eqref{eq:concaveright}  we have 
\begin{align}
L(k)-R(k) & =	4k^2 (\beta^2_k-\beta_{k-1}\beta_{k+1})
+a(\beta^2_{k-1}-\beta_{k-2}\beta_{k}) 
+b(\beta_{k-1}\beta_{k}-\beta_{k-2}\beta_{k+1}) \notag\\
& \quad \quad +4(\beta_{k-2}\beta_{k}+\beta_{k-1}\beta_{k+1}-2\beta_{k-2}\beta_{k+1}) \notag\\
& > 0. \label{eq:concave:final}
\end{align}
To prove~\eqref{eq:concave:final}, by using~\eqref{eq:concave:proof:ia} twice with $t=k-1$ and $t=k$ we have
\begin{equation}
\label{eq:concave:final:fst}
\beta^2_{k-1}> \beta_{k-2}\beta_{k}~~~\mbox{and}~~
\beta^2_k> \beta_{k-1}\beta_{k+1}. 
\end{equation}
Next, combining~\eqref{eq:concave:four-number} and~\eqref{eq:concave:final:fst} shows\begin{equation}
\label{eq:concave:final:snd}
\beta_{k-1}\beta_{k}\ge \beta_{k-2}\beta_{k+1}~~\mbox{and}~~ 
\beta_{k-2}\beta_{k}+\beta_{k-1}\beta_{k+1}\ge 2\beta_{k-2}\beta_{k+1}. 
\end{equation}
Then~\eqref{eq:concave:final} follows from~\eqref{eq:concave:final:fst}-\eqref{eq:concave:final:snd}  and the fact that $a>0$ and $b>0$. 
This completes the proof of~\eqref{eq:pdf:lc:induction}, and hence also the theorem.
\epf
\end{proof}


Note that the last statement of Theorem~\ref{them:log-concav} does not hold for $n=8$ because in this case we have $\nabla(n)=2$ and $\pmf_n(2)=\pmf_n(3)=16/33$.

\twu{
Finally, we show that there exists a unique change point between $\ypmf_n$ and $\pmf_n$. Note that a similar result for $\ypmf^*_n(k)$ and $\pmf^*_n(k)$ is established in~\citet[Theorem 8]{WuChoi16}.  
}
\begin{theorem}
\label{thm:change:point}
Let $n\ge 6$.  Then the ratio $\ypmf_n(k)/\pmf_n(k)$ is strictly increasing for $2 \leq k \leq n/2$.  In particular, there exists a real number $\kappa_n$ with \twu{$2  < \kappa_n < n/2$} such that
$$
\ypmf_n(k)<\pmf_n(k)~\mbox{for $2\le k < \kappa_n$} \quad \mbox{and} \quad
\ypmf_n(k)>\pmf_n(k)~\mbox{for $\kappa_n < k \le n/2$}.
$$
\end{theorem}

\begin{proof}
First, by~\eqref{eq:yhk_cherry_probability} in Theorem~\ref{thm:yule_cherry_probability_distribution} we have 
\begin{equation}
\label{eq:chpoint:yule}
\ypmf_{n+1}(t)=\frac{2t}{n}\ypmf_n(t)+\frac{n-2t+2}{n}\ypmf_n(t-1)~~
\mbox{ for $n\ge 5$ and $2\le t < n$.} 
\end{equation}

Next,  we have $\ypmf_n(2)>\pmf_n(2)$ for $n\ge 6$ because $\ypmf_5(2)=1=\pmf_5(2)$ and
$$
\frac{\ypmf_{n}(2)}{\ypmf_{n-1}{(2)}}=\frac{4}{n-1}<\frac{n}{2n-5}=\frac{\pmf_{n}(2)}{\pmf_{n-1}(2)}
~\mbox{for $n\ge 6$,}
$$
 where the first equality derives from~\eqref{eq:chpoint:yule} and $\ypmf_n(1)=0$, and the second equality follows from~\eqref{eq:pmf}.


Finally, when $n\ge 6$, since the sum of $\ypmf_n(k)$ (resp. $\pmf_n(k)$) over $k$ between 2 and $n/2$  is 1 and we have  $\ypmf_n(2)>\pmf_n(2)$, 
 it remains to establish the inequality in 
\begin{equation}
    f_n(k):=\frac{\ypmf_n(k-1)}{\ypmf_n(k)} <
    \frac{\pmf_n(k-1)}{\pmf_n(k)}=\frac{4k (k - 2)}{(n - 2k + 2) (n - 2k + 1)}:=g_n(k)    \notag
\end{equation}
for $3\le k\le n/2$. 
To this end, we shall proceed by induction on $n$. 
The base case $n=6$ follows from  $f_6(3)=4<6=g_6(3)$ in view of~\eqref{eq:six:distribution}. 
Let $n\ge 6$ and assume that $f_n(k)<g_n(k)$ holds for $3\le k\le n/2$. For  the induction step it suffices to show
\begin{equation}
\label{eq:change:point:induction}
f_{n+1}(k)<g_{n+1}(k)
~~\mbox{for $3\le k\le (n+1)/2$.}
\end{equation}
 
In the remainder of the proof, we shall establish~\eqref{eq:change:point:induction} by considering the following three cases. The first one is $k=3$. Since $f_{n+1}(3)>0$ and $g_{n+1}(3)>0$, the inequality in~\eqref{eq:change:point:induction} follows from
 \begin{eqnarray*}
 	\frac{1}{f_{n+1}(3)}=\frac{\ypmf_{n+1}(3)}{\ypmf_{n+1}(2)}=\frac{3\ypmf_n(3)}{2\ypmf_n(2)}+\frac{n-4}{4}
 	> 	\frac{3(n-4)(n-5)}{24}+\frac{n-4}{4} 
 	>\frac{(n-3)(n-4)}{12}=\frac{1}{g_{n+1}(3)},
 \end{eqnarray*}	
\noindent
where the second equality derives from using~\eqref{eq:chpoint:yule} twice with $t=3$, and with $t=2$ and $\ypmf_n(1)=0$;  the first inequality follows from $0<f_n(3)<g_n(3)=12/(n-4)(n-5)$, as implied by the induction assumption. 

The second case\footnote{Note that in this case the induction assumption is not applicable to  $f_n(k)$ and $g_n(k)$ as $k>n/2$ implies $f_n(k)=g_n(k)=0$.} occurs when $n+1$ is an even number and $k=(n+1)/2$. Using $n-2k+1=0$, we have $g_{n+1}(k)=2k(k-2)$ and $3g_n(k-1)=2(k-1)(k-3)$. Therefore~\eqref{eq:change:point:induction} follows from
$$
	f_{n+1}(k)=2(k-1)+3f_{n}(k-1)<2(k-1)+2(k-1)(k-3)=2(k-1)(k-2)<2k(k-2)=g_{n+1}(k), 
$$
where the first equality derives from using~\eqref{eq:chpoint:yule} twice with $t=k-1$, and with $t=k$ and $\ypmf_n(k)=0$; the first inequality follows from $3f_n(k-1)<3g_n(k-1)=2(k-1)(k-3)$, as implied by the induction assumption. 

The final case is  $3<k<(n+1)/2$. Since $3<k\le n/2$, we have
\begin{eqnarray*}
f_{n+1}(k)&=&\frac{n\ypmf_{n+1}(k-1)}{n\ypmf_{n+1}(k)}
	=\frac{2(k-1)\ypmf_n(k-1)+(n-2k+4)\ypmf_n(k-2)}{2k\ypmf_n(k)+(n-2k+2)\ypmf_n(k-1)} \\
	&=&\frac{2(k-1)+(n-2k+4)f_n(k-1)}{(n-2k+2)+\frac{2k}{f_n(k)}} < \frac{2(k-1)+(n-2k+4)g_n(k-1)}{(n-2k+2)+\frac{2k}{g_n(k)}},
\end{eqnarray*}
where the second equality derives from using~\eqref{eq:chpoint:yule} twice with $t=k$ and $t=k-1$; the third equality follows from $\ypmf_n(k-1)>0$. Furthermore, the inequality derives from $f_n(k-1)<g_n(k-1)$ and $0<f_n(k)<g_n(k)$, both following from the induction assumption. Since $g_n(k)>0$, for 
\eqref{eq:change:point:induction} it suffices to show 
\begin{equation}
\label{eq:chpoint:last}
2(k-1)+(n-2k+4)g_n(k-1) < (n-2k+2)g_{n+1}(k)+\frac{2kg_{n+1}(k)}{g_n(k)}.
\end{equation}
To this end, denote the left-hand and right-hand sides of~\eqref{eq:chpoint:last} by $L(k,n)$ and $R(k,n)$, respectively. Since
$$
\frac{g_{n+1}(k)}{g_n(k)}=\frac{n-2k+1}{n-2k+3}=1-\frac{2}{n-2k+3}
$$
and
$$
(n-2k+2)g_{n+1}(k)-(n-2k+4)g_n(k-1)=\frac{4k(k-2)}{n-2k+3}-\frac{4(k-1)(k-3)}{n-2k+3}=\frac{4(2k-3)}{n-2k+3},
$$
we have
$$
R(k,n)-L(k,n)=2k-\frac{4k}{n-2k+3}+\frac{4(2k-3)}{n-2k+3}-2(k-1)=2+\frac{4(k-3)}{n-2k+3}>0. 
$$
This completes the proof of~\eqref{eq:chpoint:last}, hence also the theorem. 
\epf
\end{proof}

\subsection{Total Variation Distance}

In this subsection, we study the differences between rooted and unrooted cherry distributions for the two models. One common approach to quantifying such differences \rev{is} total variation distance, that is, the largest possible difference between the probabilities that the two probability distributions can assign to the same event. More specifically, we are interested in the behaviour of
$$
 d^u_{TV}(\ch^*_n, \ch_n) = \frac{1}{2} \sum_{k = 1}^{\lfloor n/2 \rfloor} |\pmf^*_n (k) - \pmf_n (k)|
 ~\quad~
 \mbox{and}
  ~\quad~
   d^y_{TV}(\ch^*_n, \ch_n) = \frac{1}{2} \sum_{k = 1}^{\lfloor n/2 \rfloor} |\ypmf^*_n (k) - \ypmf_n (k)|.
$$

We begin with the total variation distance $d^u_{TV}(\ch^*_n, \ch_n)$ for the PDA model. To this end,
recall that $\Delta(n)=\lfloor n/4 \rfloor$.

\begin{lemma}
\label{pos}
Let $n \ge 4$. Then we have
\begin{equation}
    \label{eq:pmf:rpmf}
    \pmf^*_n(k) = \frac{(2n-3-2k)}{(2n-3)}\pmf_n(k) + \frac{2(k+1)}{(2n-3)} \pmf_n(k+1), \quad 1 \le k \le n/2.
\end{equation}
Moreover, we have  $\pmf^*_n(k) \ge \pmf_n(k)$ for $ 1 \le k \le \Delta(n)$, and $\pmf^*_n(k) < \pmf_n(k)$ for $ \Delta(n) < k \le n/2.$
\end{lemma}

\begin{proof}
Fix $n\ge 4$ and $1\le k \le n/2$. Noting that $\pmf_n(1)=0$ \twu{and $\pmf_n(k+1)=0$ for $k=\lfloor n/2 \rfloor$}, by~(\ref{eq:pmf})  and~(\ref{eq:rpmf}) we have
$$
\frac{\pmf_n(k)}{\pmf^*_n(k)} =\frac{2(2n-3)(k-1)}{(n-2)(n-3)}
~~\mbox{and}~~~
\frac{\pmf_n(k+1)}{\pmf^*_n(k)} =\frac{(n-2k)(n-2k-1)(2n-3)}{2(k+1)(n-2)(n-3)}.
$$
 This implies
$$
\frac{(2n-3-2k)}{(2n-3)}\frac{\pmf_n(k)}{\pmf^*_n(k)}+\frac{2(k+1)}{(2n-3)}\frac{\pmf_n(k+1)}{\pmf^*_n(k)}=1,
$$
from which~(\ref{eq:pmf:rpmf}) follows.

Putting $c_n={n!(n-2)! (n-4)!}/{(2n-4)!}$, then we have
\begin{equation*}
    \pmf_n(k) = c_n \frac{2^{n-2k}}{(n-2k)! k! (k-2)!}.
\end{equation*}
Using~(\ref{eq:pmf:rpmf}), we have
\begin{eqnarray}
\pmf^*_n (k) - \pmf_n (k) &=&  \frac{2}{(2n - 3)} \Big( (k + 1) \pmf_n (k+1) - k \pmf_n (k) \Big) \label{eq:diff:rt:step}\\
&=& \frac{ c_n 2^{n - 2k - 1} }{(2n - 3) (n - 2k)! k! (k - 1)!} \Big( (n - 2k)(n - 2k - 1) - 4 k(k - 1) \Big) \nonumber \\
&=& \frac{ c_n 2^{n - 2k - 1} }{(2n - 3) (n - 2k)! k! (k - 1)!} \Big( n(n - 1) - 2 (2n - 3) k \Big).
\nonumber
\end{eqnarray}
Hence, $\pmf^*_n (k) - \pmf_n (k) \ge 0$ if and only if $k \le \frac{ n(n - 1) }{ 2(2n - 3) }$ or, equivalently, $1 \le k \le \Delta(n)$ in view of~(\ref{eq:delta}).
\epf
\end{proof}

\begin{theorem} \label{thm:pda_cherry_dtv}
For $n\ge 4$, the total variation distance $d^{u}_{TV}(\ch^*_n, \ch_n)$ between the cherry distributions under the PDA model is
\begin{equation}
\label{eq:pda_cherry_dtv}
  d^{u}_{TV}(\ch^*_n, \ch_n) = \frac{n! (n-2)! (n-4)! 2^{n-2\Delta(n)-1}}{(2n-3)! (n-2\Delta(n)-2)! \Delta(n)! (\Delta(n)-1)!} =\frac{1}{\sqrt{2 \pi n}} \left(1+o(1)\right).   \end{equation}
\end{theorem}

\begin{proof}
Write $x_+ = \max\{x, 0\}$. Noting that $\sum_{k = 1}^{\lfloor n/2 \rfloor} \left(\pmf^*_n (k) - \pmf_n (k) \right)=0$, 
we have
\begin{eqnarray}
 d^u_{TV}(\ch^*_n, \ch_n) &=& \frac{1}{2} \sum_{k = 1}^{\lfloor n/2 \rfloor} |\pmf^*_n (k) - \pmf_n (k)| = \sum_{k = 1}^{\lfloor n/2 \rfloor} \Big( \pmf^*_n (k) - \pmf_n (k) \Big{)}_+  \notag \\
&=& \sum_{k = 1}^{\Delta(n)} \Big( \pmf^*_n (k) - \pmf_n (k) \Big) \label{eq:pda_cherry_dtv:pf:fst} \\
&=& \frac{2}{2n - 3} \sum_{k = 1}^{\Delta(n)} \Big( (k + 1) \pmf_n (k + 1) - k \pmf_n (k) \Big) \label{eq:pda_cherry_dtv:pf:snd} \\
&=& \frac{2 \, (\Delta(n) + 1) \, \pmf_n (\Delta(n) + 1)}{2n - 3}, \notag
 \end{eqnarray}
\twu{where the equality in~\eqref{eq:pda_cherry_dtv:pf:fst} follows from Lemma~\ref{pos} and the  equality in~\eqref{eq:pda_cherry_dtv:pf:snd} follows from~\eqref{eq:diff:rt:step}.}
\twu{This establishes} the first equality in~(\ref{eq:pda_cherry_dtv}) 
after substituting the expression of $\pmf_n(\Delta(n)+1)$.

\rev{
To prove the second equality in~(\ref{eq:pda_cherry_dtv}),
we abbreviate $\Delta(n)=\lfloor n/4 \rfloor$ to $\Delta$ for simplicity. Note first that $\Delta=\frac{n}{4}(1+o(1))$. Furthermore, we have
\begin{equation}
\label{eq:delta:approx}
g_n:=\left(\frac{2n-4\Delta}{n}\right)^{n-2\Delta}\left(\frac{4\Delta}{n}\right)^{2\Delta+1}=1+o(1).
\end{equation}
To prove~(\ref{eq:delta:approx}), for $r\in \{0,1,2,3\}$ let $\{g^r_{m}\}_{m\ge 1}$  be the subsequence of $\{g_n\}_{n\ge 1}$ consisting of entries $g_n$ with $n=4m+r$ for some $m\ge 1$. Note that each item in the subsequence $g^0_{m}$ is 1. Furthermore, for $r=1$ we have
$$
\lim_{m\to \infty}g^1_{m}=\lim_{m\to \infty}\left(\frac{4m+2}{4m+1} \right)^{2m+1}\left(\frac{4m}{4m+1} \right)^{2m+1}
=\lim_{m\to \infty}\left(1-\frac{1}{(4m+1)^2} \right)^{2m+1}=1.
$$
Using a similar argument, we have $g^r_{n}=1+o(1)$ for $r\in \{2,3\}$. Therefore we can conclude that $g_n=1+o(1)$.
}

\rev{
Now we have
\begin{eqnarray}
	\frac{2 \, (\Delta + 1) \, \pmf_n (\Delta + 1)}{2n - 3} &=& \frac{n! \, (n - 2)! \, (n - 4)! \, 2^{n - 2\Delta - 1}}{(2n - 3)! \, (n - 2\Delta - 2)! \, \Delta! \, (\Delta - 1)!} \notag \\
    &=& \frac{(n!)^3 \, 2^{n - 2\Delta - 1 } }{ (2n)! \, (n - 2\Delta)! \, (\Delta!)^2} \times \frac{2n(2n-1)(2n-2)(n-2\Delta)(n-2\Delta-1)\Delta}{n(n-1)n(n-1)(n-2)(n-3)} \notag \\
	&=& \frac{(n!)^3 \, 2^{n - 2\Delta - 1 } }{ (2n)! \, (n - 2\Delta)! \, (\Delta!)^2} \times \frac{4(2n-1)(n-2\Delta)(n-2\Delta-1)\Delta}{n(n-1)(n-2)(n-3)} \notag \\
	&=& \frac{(n!)^3 \, 2^{n - 2\Delta - 2 } }{ (2n)! \, (n - 2\Delta)! \, (\Delta!)^2} \Big( 1 + o(1) \Big), 
	 \label{eq:pda_cherry_dtv:middle}
\end{eqnarray}
\twu{where the last equality follows from $\Delta=\frac{n}{4}(1+o(1))$.} Using Stirling's formula:
$
n!=\sqrt{2\pi n} (n/e)^n(1+o(1)),
$
~\citep[see, e.g.][Chapter 6]{abramowitz1964handbook}, the expression in~\eqref{eq:pda_cherry_dtv:middle} can be further simplified as
\begin{eqnarray*}
\frac{(n!)^3 \, 2^{n - 2\Delta - 2 } }{ (2n)! \, (n - 2\Delta)! \, (\Delta!)^2} \Big( 1 + o(1) \Big)
        &=& \frac{n^{n+1} (1+o(1))}{\sqrt{\pi(n-2\Delta)}\,2^{n+2\Delta+3} (n-2\Delta)^{n-2\Delta} \Delta^{2\Delta+1} } \\
                &=&\frac{1+o(1)}{g_n \sqrt{\pi(n-2\Delta)}} \\
    &=& \frac{1}{\sqrt{2 \pi n}} \Big( 1 + o(1) \Big).
\end{eqnarray*}
Here 
 the last equality follows from~(\ref{eq:delta:approx}).
This completes the proof of the second equality in~(\ref{eq:pda_cherry_dtv}).
}
\epf
\end{proof}


To study the total variation distance between the cherry distributions under \rev{the YHK} model, we need some further observations. Note that for $n\ge 4$, we have
 \begin{equation}
 \label{eq:k:1}
\ypmf^*_n(1)=\frac{2^{n-2}}{(n-1)!}>\ypmf_n(1)=0
 \end{equation} 
because there \rev{exist} rooted trees in $\rtsp_n$ with one cherry, \twu{whereas} each tree in $\tsp_n$ has at least two cherries.
   On the other hand, putting $m=\lfloor n/2 \rfloor$, then we have $\sum_{k=1}^m\ypmf^*_n(k)=\sum_{k=1}^m\ypmf_n(k)=1$. Therefore, we have the following result showing that there exists a sign change between the point-wise difference between these two functions.

\begin{lemma} \label{lem:yhk_cherry_dtv_k0}
For each $n \ge 4$, there exists an integer $1<k_0=k_0(n)\le n/2$ with
\begin{equation*}
   \left (\ypmf^*_n(k_0)-\ypmf_n(k_0) \right) \times  \left (\ypmf^*_n(k_0-1)-\ypmf_n(k_0-1) \right)< 0.
\end{equation*}
\end{lemma}

\begin{figure}
	\centering

    	\includegraphics[width=0.6\linewidth]{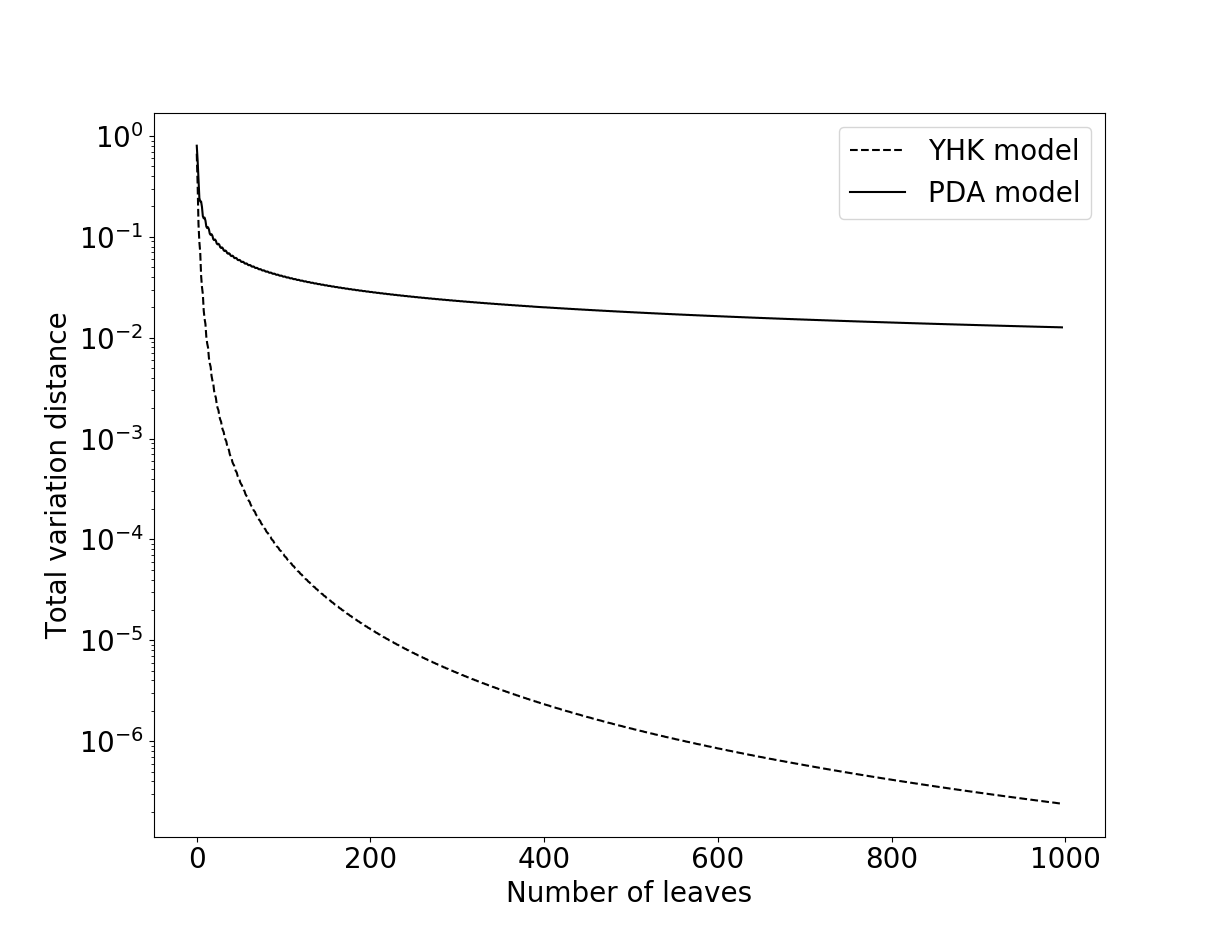}

    \caption{Total variation distances between the cherry distributions of rooted and unrooted trees ($4 \leq n \leq 1000$) under the YHK model (dashed line) and the PDA model (solid line).}
    \label{fig:cherry_dtv}
\end{figure}

\rev{With Lemma~\ref{lem:yhk_cherry_dtv_k0}, we are in a position to prove the following result.}


\begin{proposition}
\label{prop:yule:tv}
The sequence of the total variation distances $\{d^y_{TV} (\ch^*_n, \ch_n)\}_{n\ge 4}$ between the cherry distributions under the YHK model is a strictly decreasing sequence in $n$.
\end{proposition}

\begin{proof}
\twu{For $1\le k \le 2n$, let} $\alpha_n(k)=2k/n$ and $\beta_n(k)=(n-2k+2)/n$.
\rev{By~(\ref{eq:yhk_cherry_probability}) and \citet[Eq.~(11)]{WuChoi16}} we have
\begin{eqnarray*}
    \ypmf^*_{n+1}(k) = \alpha_n(k)  \ypmf^*_{n}(k) + \beta_n(k) \ypmf^*_{n}(k-1)~~\mbox{and}~~
     \ypmf_{n+1}(k) = \alpha_n(k)  \ypmf_{n}(k) + \beta_n(k) \ypmf_{n}(k-1), 
\end{eqnarray*}
\twu{for $n\ge 4$ and $1\le k \le n/2$, where the case $k=1$ follows from~\eqref{eq:k:1} and the fact that $\ypmf_n(0)=\ypmf_n(1)=\ypmf^*_n(0)=0$.}

For simplicity, set $m=\big \lfloor \frac{n+1}{2} \big \rfloor$. Then we have
$\ypmf^*_{n+1}(k)=\ypmf_{n+1}(k)=0$ for $k>m$. Therefore we have
{\small
\begin{equation}
\label{eq:tv:step}
2d^y_{TV} (\ch^*_{n+1}, \ch_{n+1})=\sum_{k=1}^m |\ypmf^*_{n+1}(k)-\ypmf_{n+1}(k)|
=\sum_{k=1}^m \left| \alpha_n(k)  \left(\ypmf^*_{n}(k)-\ypmf_{n}(k)\right) + \beta_n(k) \left(\ypmf^*_{n}(k-1)- \ypmf_{n}(k-1) \right) \right|.
\end{equation}
}

By Lemma~\ref{lem:yhk_cherry_dtv_k0}, let $k_0$ be a constant so that  $\left (\ypmf^*_n(k_0)-\ypmf_n(k_0) \right) \times \left (\ypmf^*_n(k_0-1)-\ypmf_n(k_0-1) \right)< 0$ holds.
By the triangle inequality, using~(\ref{eq:tv:step}) it follows that
\begin{eqnarray*}
2d^y_{TV} (\ch^*_{n+1}, \ch_{n+1}) &<&
\sum_{k\not = k_0} \left| \alpha_n(k)  \left(\ypmf^*_{n}(k)-\ypmf_{n}(k)\right) \right| +
\sum_{k\not = k_0} \left| \beta_n(k) \left(\ypmf^*_{n}(k-1)- \ypmf_{n}(k-1) \right) \right| \\
&&\,\,+\alpha_n(k_0) |\left(\ypmf^*_{n}(k_0)-\ypmf_{n}(k_0)\right)|
+\left| \beta_n(k_0) \left(\ypmf^*_{n}(k_0-1)- \ypmf_{n}(k_0-1) \right) \right| \\
&=&\sum_{k=1}^m  \alpha_n(k)  \left| \left(\ypmf^*_{n}(k)-\ypmf_{n}(k)\right) \right| +
\sum_{k=1}^m  \beta_n(k) \left|\left(\ypmf^*_{n}(k-1)- \ypmf_{n}(k-1) \right) \right|  \\
&=&\sum_{k=1}^m  \alpha_n(k)  \left| \left(\ypmf^*_{n}(k)-\ypmf_{n}(k)\right) \right| +
\sum_{k=0}^{m-1}  \beta_n(k+1) \left|\left(\ypmf^*_{n}(k)- \ypmf_{n}(k) \right) \right| \\
&=&\sum_{k=1}^m  |\ypmf^*_{n}(k)-\ypmf_n(k)| \\
&=& 2d^y_{TV} (\ch^*_{n}, \ch_{n}),
\end{eqnarray*}
from which \rev{this} proposition holds. Note that the second last equality holds because $\ypmf^*_{n}(0)=\ypmf_n(0)=0$, $\alpha_n(m)=1$ when $n$ is even, and $\ypmf^*_{n}(m)=\ypmf_n(m)=0$ when $n$ is odd.
\epf
\end{proof}

\section{Discussion and Conclusion} \label{discussion}

Tree shape statistics play an important role in studying evolutionary signals in  phylogenetic trees, so it is helpful to understand how they are related between tree generating models for rooted and unrooted trees. In this paper, we present a comparison study on properties of statistical distributions for
cherries and pitchforks under the YHK and the PDA models. In addition to common patterns between rooted and unrooted \rev{trees} for both models, such as the log-concavity of the cherry distributions, we also observe some differences.  For instance, by Proposition~\ref{prop:mean:comp} we know that the difference between the mean number of cherries (resp. pitchforks) for unrooted trees and \twu{for} rooted trees converges to $0$ under the YHK model \rev{but} to $1/4$ under the PDA model.
 \rev{As a result, due caution is required for conducting statistical analysis for datasets containing both rooted and unrooted trees: when subtree statistics are computed from such a dataset, 
 bias could be introduced if we simply treat the rooted trees as unrooted ones by ignoring their roots. }

Several questions derived from the work presented here remain open.
For example, numerical computation (see, e.g. Fig.~\ref{fig:cherry_dtv}) suggests that the total variation distance $d^y_{TV}(\ch^*_n, \ch_n)$ is bounded above by $d^u_{TV}(\ch^*_n, \ch_n)$. If this can be established analytically, then by Theorem~\ref{thm:pda_cherry_dtv} and Proposition~\ref{prop:yule:tv} it follows that $d^y_{TV}(\ch^*_n, \ch_n)$ also converges to zero. Next, we conjecture that the pitchfork distributions for both rooted trees and unrooted trees are log-concave under the two null models. Note that log-concave and unimodal sequences arise naturally from problems in a variety of fields~\citep{stanley1989log}, including in phylogenetics~\citep{zhu2014clades,WuChoi16}.
\rev{Furthermore,  for rooted trees, previous studies have investigated  various properties of subtrees with four or more leaves, including mean, variance, and asymptotic distribution ~\citep[see, e.g.][]{rosenberg06a,chang2010limit,Janson2014}, but much less is known for unrooted trees.}

The work presented here also leads to some broad questions that may be interesting to explore in future work. First, the two models considered here can be regarded as two special cases for some more general tree generating models,  such as Ford's alpha model~\citep{chen2009new} and the Aldous $\beta$-splitting model~\citep{aldous96a}. Therefore it is of interest to extend our studies on subtree indices to these models as well.
\twu{Secondly, our results indicate that the problem of comparing distributions of shape statistics between rooted and unrooted trees is nontrivial.} 
Finally, one can also consider aspects of tree shapes that are related to the distribution of branch lengths~\citep{Ferretti2017,arbisser2018joint}, which will help us design more complex models that may in some cases provide a better fit to real data.

\begin{acknowledgements}
The authors would like to sincerely thank two anonymous referees for their insightful comments and constructive suggestions. 
\end{acknowledgements}

%
%

\newpage

 \newcommand{\noop}[1]{}

\end{document}